\documentclass[11pt, leqno]{amsart}
\usepackage{amssymb,amscd,amsthm,amsxtra}
\usepackage{latexsym}
\usepackage[mathscr]{euscript}
\usepackage{amsfonts}
\usepackage{color}
\usepackage[colorlinks=true]{hyperref}
\usepackage{mathrsfs}
\makeatletter
\newtheorem{theorem}{Theorem}[section]

\newtheorem{definition}[theorem]{Definition}
\newtheorem{lemma}[theorem]{Lemma}
\newtheorem{proposition}[theorem]{Proposition}
\newtheorem{remark}[theorem]{Remark} 

\numberwithin{equation}{section}
\begin{document}
\title[New Problem]{Linear non-degeneracy and uniqueness of the bubble solution for the critical fractional H\'enon equation in $\mathbb{R}^N$}
\author[Salom\'on Alarc\'on]{S. Alarc\'on$^{\dag,1}$}
\author[Bego\~na Barrios]{B. Barrios$^{\ddag,2}$}
\author[Alexander Quaas]{A. Quaas$^{*,1}$}

\address[$\dag$]{Departamento de  Matem\'atica, Universidad T\'ecnica Federico Santa Mar\'ia, Casilla 110-V, Valpara\'iso, Chile.}
\address[$\ddag$]{Departamento de An\'alisis Matem\'atico, Universidad de La Laguna C/. Astrof\'isico Francisco S\'anchez s/n, 38200 - La Laguna, Spain.}
\email[$1$]{salomon.alarcon@usm.cl}
\email[$2$]{bbarrios@ull.es}
\email[$3$]{alexander.quaas@usm.cl}
\date{\today}
\maketitle
\begin{abstract}
{{We study the equation 
\begin{equation*}\label{P0}
(-\Delta)^s u = |x|^{\alpha} u^{\frac{N+2s+2\alpha}{N-2s}}\mbox{ in }\mathbb{R}^N,\tag{P}
\end{equation*}
where  $(-\Delta)^s$ is the fractional Laplacian operator with $0 < s < 1$, $\alpha>-2s$ and  $N>2s$. We prove the linear non-degeneracy of positive  radially symmetric  solutions of the equation (\ref{P0}) and, as a consequence, a uniqueness result of those solutions with Morse index equal to one. In particular, the ground state solution is unique.   Our non-degeneracy result extends in the radial setting some known theorems done by D\'avila, Del Pino and Sire (see \cite[Theorem 1.1]{Davila-DelPino-Sire}), and Gladiali, Grossi and Neves (see \cite[Theorem 1.3]{Gladiali-Grossi-Neves}).}}
\end{abstract}


\setcounter{equation}{0}
\section{Introduction}
Ground state solutions play a fundamental role in analysis and mathematical physics in general, and are usually obtained by minimizing an energy functional on a manifold. 
Due to the Euler-Lagrange conditions, these are sometimes weak  positive solutions of a nonlinear elliptic equation with Morse Index equal to one. 
In the second-order one-dimensional case the ground state is unique and corresponds to the homoclinic given by the Poincar\'e-Bendixson Theorem \cite{Cod}. 
In analysis, such solutions are connected with optimal constants in Hardy-Littlewood \cite{ha}\cite{ha1} -Sobolev \cite{S} type inequalities and many other problems as we comment below.\bigskip

The study of the uniqueness and the non-degeneracy property is a cornerstone issue in many problems in Partial Differential Equations. 
For example, some concentration phenomena may be studied via the Lyapunov-Schmidt reduction method that is based on the linear non-degeneracy and uniqueness up to translation and/or rescaling of the solutions. 
The linear non-degeneracy is also connected with the stability of solutions for some nonlinear parabolic equations and for establishing bifurcation results to several equations (\cite{Roy}, \cite{Ireneo}).  
In particular, bifurcation of non-radial solutions at the H\'enon critical equation is found  in \cite{Gladiali-Grossi-Neves}. 
For the fractional nonlinear  Schr\"odinger equation, a review and many connections with analysis and mathematical physics can be found in the work by Frank and Lenzmann \cite{Frank-Lenzmann}, Frank, Lenzmann and Silvestre \cite{Frank-Lenzmann-Silvestre} (see also \cite{Felmer}). 
To mention other application, D\'avila, Del Pino and Wei in \cite{Davila-DelPino-Wei} used the linear non-degeneracy to establish the existence of concentrating standing wave solutions for the fractional nonlinear Schr\"odinger equation.\bigskip

The first aim of this work is to study the linear non-degeneracy of {{positive radially symmetric solutions}} for the fractional equation
\begin{equation}\label{P}
(-\Delta)^s u = |x|^{\alpha} u^{p_{\alpha,s}^*}\quad \mbox{ in }\mathbb{R}^N,
\end{equation} 
where $p_{\alpha,s}^*:= 2^*_{s,\alpha} - 1= \frac{N+2s+2\alpha}{N-2s}$, $0<s<1$, $\alpha>-2s$ and $N>2s$. Here $(-\Delta)^s$, $0<s<1$, denotes the well-known \emph{fractional Laplacian}, that is defined on smooth functions as {
\begin{equation}\label{operador}
(-\Delta)^s u(x) = C_{N,s}\,\mathrm{P.V.} \int_{\mathbb{R}^N} \frac{u(x)-u(y)}{|x-y|^{N+2s}} dy,\qquad x\in \mathbb{R}^N,
\end{equation}
where $C_{N,s}$ is a normalization constant and P.V. means the principal value that is we omitted for brevity}. The integral in \eqref{operador} 
has to be understood in the principal value sense, that is, as the limit as $\epsilon\to 0$ of the 
same integral taken in {$\mathbb{R}^N\setminus B_\epsilon(x)$, i.e, the complementary of the ball of center $x$ and radius $\epsilon$}.  It is clear that the fractional Laplacian operator is well defined for functions that belong, for instance, to $\mathcal{L}_{2s}(\mathbb{R}^N)\cap \mathcal{C}^{1,1}_{loc}$ where
$$\mathcal{L}_{2s}(\mathbb{R}^N):=\{u:\mathbb{R}^{N}\to \mathbb{R}:\, \int_{\mathbb{R}^{N}}\frac{|u(x)|}{1+|x|^{N+2s}}<\infty\}.$$

In the case $s=1$ and $\alpha>-2$, the equation (\ref{P}) corresponds to the critical H\'enon equation in $\mathbb{R}^N$ (see \cite{He, Gladiali-Grossi-Neves}). {{Since  there are no solutions to the equation \eqref{P}  in $\mathcal{L}_{2s}(\mathbb{R}^N)$ with $p \in (1,p_{\alpha,s}^*)$ instead $p_{\alpha,s}^*$  (see \cite[Theorem 1.1]{Barrios-Quaas}), then (\ref{P}) is usually called the critical fractional H\'enon equation.}}
\bigskip

The critical Lane-Emden-Folwer equation (when $\alpha=0$) {{and in general the critical}} H\'enon equation (case $\alpha\neq 0$) are connected with optimal constants in some Hardy-Littlewood-Sobolev type inequalities, mass transport, and concentration phenomena, see \cite{Davila-DelPino-Sire, Gladiali-Grossi-Neves, Barrios-Quaas}, and the references therein. When $\alpha=0$ see also the works of Aubin \cite{Au}, Talenti \cite{CGS} and Caffarelli, Gidas and Spruck \cite{CGS} for the case $s=1$, and those of Lieb (\cite{Lieb}), Carlen and Loss (\cite{CL}), Frank and Lieb (\cite {FL}, \cite{FL1}), Chen, Li and Ou (\cite{CLO}), Li (\cite{Li}), and Li and Zhu (\cite{LZ}) when $0<s<1$.

\bigskip

To start with our main results, let $U_{\alpha,s}$ be the {{positive radially symmetric}} bubble solution of (\ref{P}) found in \cite{Barrios-Quaas} or any other positive radially symmetric solution of (\ref{P}). 
Henceforth we will assume  
$$
\left\{\begin{array}[c]{lll}
-2s <\alpha &\quad \mbox{if } \frac{1}{2}\leq s<1,\medskip\\
-2s<\alpha< \frac{2s(N-1)}{1-2s} &\quad \mbox{if }0<s<\frac{1}{2}.
\end{array}\right.
$$ 
The existence results in \cite{Barrios-Quaas} require that in the case $0<s<\frac{1}{2}$,  $p^*_{\alpha,s}$ is below the critical exponent in dimension one, that is possible only if $\alpha$ satisfies the extra condition $-2s<\alpha< \frac{2s(N-1)}{1-2s}$.
Existence or nonexistence in the case $\alpha\geq \frac{2s(N-1)}{1-2s}$ is still nowadays an open problem.\bigskip 

Consider now the linear operator associated to problem (\ref{P}), given by  
\begin{equation}\label{Ele}
\mathfrak{L}_{U_{\alpha,s}}u:=((-\Delta)^s{+}V)u=(-\Delta)^s u - p_{\alpha,s}^* |x|^{\alpha} U_{\alpha,s}^{p_{\alpha,s}^*-1}u
\end{equation}
where 
\begin{equation}\label{Uve} 
V=V(x):={-p_{\alpha,s}^* |x|^{\alpha}} U_{\alpha,s}^{p_{\alpha,s}^*-1}\leq 0,\quad x\in\mathbb{R}^{N}. 
\end{equation}
Our first main results is

\begin{theorem}\label{main}
Let $0<s<1$, $\alpha>-2s$ and $N>2s$. 
The linearized operator $\mathfrak{L}_{U_{\alpha,s}}$ given in {{(\ref{Ele}) acting}} on $L^{2}(\mathbb{R}^{N})$ is non-degenerate, i.e., its kernel is given by
$$
{\rm Ker}\,\mathfrak{L}_{U_{\alpha,s}}= \,\mathrm{span} \{z\},
$$
where
\begin{equation}\label{ZZ}
z:=\frac{N-2s}{2}U_{\alpha,s} +x\cdot \nabla U_{\alpha,s}=\frac{\partial U^\lambda_{\alpha,s}}{\partial \lambda}
\end{equation}
at $ \lambda=1 $ and $U^\lambda_{\alpha,s}(x)=\lambda^{\frac{N-2s}{2}}U_{\alpha,s}(\lambda x)$.
\end{theorem}\medskip

Notice that the fact $\mathfrak{L}_{U_{\alpha,s}} z=0$ may be proved by using
$$
(-\Delta)^s (x \cdot \nabla u)=2s(-\Delta)^su+x\cdot \nabla ((-\Delta)^su).
$$ 

The previous result extends in the radial setting \cite[Theorem 1.1]{Davila-DelPino-Sire}, where the case  $0<s<1$ and $\alpha=0$ was considered, and \cite[Theorem 1.3]{Gladiali-Grossi-Neves}, where the case $s=1$ and $\alpha>0$ was studied. 
In \cite{Gladiali-Grossi-Neves},  the linear non-degeneracy for the case  $s=1$ and $\alpha=0$ was obtained by using the results of Rey given in \cite{Roy} (see \cite[Lemma 3.1.2]{Ireneo} for an alternative proof)  via a change of variable. These results can be extended  directly to the case  $s=1$ and $\alpha>-2$. 
Notice that, since our computations are stable when $s\to 1$, the same results may be obtained directly for the case $s=1$ and $\alpha>-2$.\bigskip

The main tool to prove the non-degeneracy result is a Emden-Fowler type transformation which allows us to study the problem (\ref{P})  through an alternative problem given by the following nonlinear Schr\"odinger type equation
\begin{equation}\label{transformado}
\mathcal{T}_{s}v+ \mathcal{A}_{s,N} v = v^{p^*_{\alpha,s}}\quad \mbox{ in }\mathbb{R},
\end{equation}
where the operator  $\mathcal{T}_{s}$ is given by {{
\begin{equation}\label{TT}
\mathcal{T}_{s} v(\kappa)= C_{N,s} \int_{\mathbb{R}} (v(\kappa))-v(\tau))\mathcal{K}(\kappa-\tau)\,d\tau,\quad \kappa \in\mathbb{R},
\end{equation}
with}} 
\begin{equation}\label{kernel}
\mathcal{K}(t)= \,e^{-t\frac{N+2s}{2}} \int_{\mathbb{S}^{N-1}} \frac{1}{ |1+e^{-2t} -2e^{-t}\langle \theta, \vartheta\rangle|^{\frac{N+2s}{2}}}d\vartheta,\quad t\in\mathbb{R},
\end{equation}
and the constant $\mathcal{A}_{s,N}$ is defined by{{
\begin{equation}\label{AA}
\mathcal{A}_{s,N}=C_{N,s} \,\int_0^{\infty} \int_{\mathbb{S}^{N-1}} \rho^{N-1}\frac{1-\rho^{\beta}}{|1+\rho^2-2\rho\langle \theta, \vartheta\rangle |^{\frac{N+2s}{2}}}d\vartheta\,d\rho,
\end{equation}
(see \cite{Manuel, beckner, Frank-Seiringer,  Herbst}). }}Note that  in the case $s=1$ one has $\mathcal{T}_{1}u=-u''$ and $\mathcal{A}_{1,N}=((N-2)/2)^2$, but $\mathcal{T}_{s}u\not=(-\Delta)_\mathbb{R}^su$ for $0<s<1$. 
As we mention, to prove the Theorem \ref{main} we first deal with the transformed linearized operator
\begin{equation}\label{L_transformado}
\widehat{\mathfrak{L}}_{\widehat{U}_{\alpha,s}}v:=\mathcal{T}_{s}v+\mathcal{A}_{s,N} v- p^*_{\alpha,s}\widehat{U}_{\alpha,s}^{p^*_{\alpha,s}-1}v,
\end{equation}
where $\widehat{U}_{\alpha,s}$ is the corresponding transformation of the function $U_{\alpha,s}$ under the change of variable, that is
\begin{equation}\label{U_transformada}
{U}_{\alpha,s}(r)=r^{-\frac{N-2s}{2}}\widehat{U}_{\alpha,s}(\ln\, r),\quad \mbox{for all } r>0.
\end{equation}
Therefore our real objective will be to prove the equivalent.\medskip

\begin{theorem}\label{nodegenerancia_extendido}
Let $0<s<1$, $\alpha>-2s$ and $N>2s$. 
The linearized operator $\widehat{\mathfrak{L}}_{\widehat{U}_{\alpha,s}}$ given in \eqref{L_transformado} acting on $L^{2}(\mathbb{R})$ is non-degenerate, i.e., its kernel satisfies
$$
{\rm Ker}\,\widehat{\mathfrak{L}}_{\widehat{U}_{\alpha,s}}=\mathrm{span}\{\widehat{U}'_{\alpha,s}\}.
$$
\end{theorem}\medskip

\begin{remark}\label{importante}
i) The solution found in \cite{Barrios-Quaas} has Morse index $N(\widehat{\mathfrak{L}}_{\widehat{U}_{\alpha,s}})=1$, see Lemma \ref{lanuestravale} below.

ii) \noindent It is clear, and important to bear in mind, that Theorem \ref{main} (resp. \ref{nodegenerancia_extendido}) 
it is true for every radially symmetric (resp. even) solution $v$ of \eqref{P} (resp. Q of \eqref{transformado}) 
with $N(\mathfrak{L}_{{U}_{\alpha,s}})=1$ (resp. $N(\widehat{\mathfrak{L}}_{\widehat{U}_{\alpha,s}})=1$) where 
$N(\mathfrak{L}_{{U}_{\alpha,s}})$ (resp. $N(\widehat{\mathfrak{L}}_{\widehat{U}_{\alpha,s}})$) is the Morse index of 
${\mathfrak{L}}_{{U}_{\alpha,s}}$ (resp $\widehat{\mathfrak{L}}_{{\widehat{U}}_{\alpha,s}}$) 
that gives the number 
(counting multiplicity) of strictly negative eigenvalues of ${\mathfrak{L}_{{U}_{\alpha,s}}}$ 
(resp $\widehat{\mathfrak{L}}_{\widehat{U}_{\alpha,s}}$) acting on $L^2(\mathbb{R}^{N})$ (resp. $L^2(\mathbb{R})$). 

iii) $z$ defined in (\ref{ZZ}) corresponds to the Emden-Fowler transform back of $\widehat{U}'_{\alpha,s}$.

iv) Just some days before submitting the present work, by a personal communication by one of the authors, the recent preprint \cite{Musina} was sent to us. We notice that in this work, only for the case {{$-2s<\alpha<0$, a non-degeneracy}} property was proved using a complete different techniques than the ones we developed here.

\end{remark}\medskip

We briefly comment the strategy that we follow to prove Theorems \ref{main} and Theorem \ref{nodegenerancia_extendido}. Our approach mainly follows the guidelines of \cite{Frank-Lenzmann} and \cite{Frank-Lenzmann-Silvestre}, being similar in some aspects byt, of course, it is not based on ODE arguments as occurs in the local case.\bigskip

Notice first that, since $\widehat{U}_{\alpha,s}$ is even, we start by considering the decomposition $L^{2}(\mathbb{R})=L_{even}^{2}(\mathbb{R})\oplus L_{odd}^{2}(\mathbb{R})$. 
If the Morse index $N(\widehat{\mathfrak{L}}_{\widehat{U}_{\alpha,s}})=1$, then the second eigenvalue of the linearized operator $\widehat{\mathfrak{L}}_{\widehat{U}_{\alpha,s}}$ has to be greater or equal than zero. 
The result follows because we first prove that $\widehat{U}'_{\alpha,s}$ is a simple eigenfunction in $L_{odd}^{2}(\mathbb{R})$ by using a Krein-Rutman type theorem to establish the simplicity  and a positivity argument of Perron-Frobenius type in the half-line $(0, \infty)$ directly on the transformed nonlocal operator defined in the half-line. 
We notice that this strategy {{for non-degeneracy}} does not use  semigroup properties that includes the perturbation arguments of \cite{Reed} such  as in \cite{Frank-Lenzmann} and \cite{Frank-Lenzmann-Silvestre}.\bigskip

Secondly, we prove that $L_{even}^{2}(\mathbb{R})$ does not have eigenfunctions with zero eigenvalue. 
We split the proof of this fact in two cases: i) the eigenfunction has one zero in the half-line $(0,+\infty)$, and ii) the eigenfunction has more than one zero in the half-line.
\\In the case i) by using the invariant of the equation  {{$v={\partial \widehat{U}_{\alpha,s}}/{\partial p^*_{\alpha,s}}$,}} that corresponds to a non-homogeneous solution of the linearized equation, we arrive to a contradiction with the Fredholm alternative. This invariant is used in \cite{Felmer-Quaas} to prove uniqueness of a no explicit critical exponent in the fully non-linear version {{of the equation 
\begin{equation}\label{P00} 
-\mathcal{M}(D^2u)=u^{p^*},
\end{equation}
where $\mathcal{M}$ is the local Pucci extremal operator, and $p^*$ is its critical exponent.}} Notice also that in \cite{Frank-Lenzmann} and \cite{Frank-Lenzmann-Silvestre} the invariant used for the standard nonlinear Schr\"odinger equation in $\mathbb{R}$ is the differentiation with respect to the scaling, see (4.4) and (4.5)  in \cite{Frank-Lenzmann} (resp. (7.2) and (7.3) in \cite{Frank-Lenzmann-Silvestre}), that does not hold for the equation \eqref{transformado}, except in the local case. 
This important difference comes from the fact that the equation \eqref{transformado} is really a critical exponent equation more than a subcritical  Schr\"odinger type equation. 

In the case ii) we use the idea of the Courant's nodal domain as in \cite{Frank-Lenzmann} (see also \cite{Frank-Lenzmann-Silvestre}) by using the extended equivalent local problem introduced in \cite{Caf-Sil} (see Propositions \ref{prop53} and \ref{clave}) by analyzing the singular eigenvalues with Hardy weight.\bigskip

As an application of our non-degeneracy result, we establish the following uniqueness result. 
Here we will need to assume that $N>2s$ for all $0<s\leq 1$ since we aim to do a homotopy argument up to $s=1$. 
that obligue us to consider $N\geq 3$. Recall that existence in the case $s=1$ is given in the celebrated work of E. Lieb \cite{Lieb}.\medskip

\begin{theorem}\label{unicidad1} 
{{Let $0<s<1$, $\alpha>-2s$ and $N\geq 3$. If $Q$ is a positive solution of \eqref{transformado} such that $N(\widehat{\mathfrak{L}}_{Q})=1$,  where $\widehat{\mathfrak{L}}_{Q}=\mathcal{T}_{s}+\mathcal{A}_{s,N}- p^*_{\alpha,s}Q^{p^*_{\alpha,s}-1}$, then $Q$ is unique up to translations.}}
\end{theorem}\medskip

From the proof of the previous result, one deduces that if $Q$ is even, then it is the unique solution. 
\bigskip

In the original variables the previous result can be read as\medskip 

\begin{theorem}\label{unicidad_original} 
{{Let $0<s<1$, $\alpha>-2s$ and $N\geq 3$. If $U$ is a positive radially symmetric solution of \eqref{P} such that $N({\mathfrak{L}_{U}})=1$,  where $\mathfrak{L}_{U}=(-\Delta)^s  - p_{\alpha,s}^* |x|^{\alpha} U^{p_{\alpha,s}^*-1}$,  then $U$ is unique up to scaling.}}
\end{theorem}\medskip

The proof of the Theorem \ref{unicidad1} borrows some ideas developed in \cite[Section 8]{Frank-Lenzmann-Silvestre} (see also \cite[Section 5]{Frank-Lenzmann}) by using Theorem \ref{nodegenerancia_extendido},  the Implicit Function Theorem and a global continuation property. 
Let us denote by $Q_s$ a solution of 
\begin{equation}\label{cc}
\mathcal{T}_{s}Q_s+ \mathcal{A}_{s,N} Q_s = |Q_s|^{p-1}Q_s\quad \mbox{ in }\mathbb{R},
\end{equation} 
where $p<\frac{1+2s}{1-2s}$.
Given a fixed $0<s_0<1$ and a solution $Q_{s_0}$ with $N_{even}(\widehat{\mathfrak{L}}_{Q_{s_0}})=1$ the scheme of the strategy will be
\begin{itemize}
\item To construct a locally unique branch $Q_s$ for $s$ {{close to $s_0$, with $s_0\leq s<1$}}.
\item To continue the local branch until $s=1$ as long as $Q_{s_{0}}>0$ by using a priori bounds.
\item To get the existence and uniqueness of the branch $Q_s$, that starts from a solution $Q_{s_0}$ by using the global uniqueness and non-degeneracy of the problem with $s=1$. We notice that the non-degeneracy in the local case comes from the results of \cite{Gladiali-Grossi-Neves} (that hold also for $\alpha>-2$) and also by the use of the uniqueness Poincar\'e-Bendixson Theorem and our non-degeneracy result (valid for $s=1$ yet).
\end{itemize}\medskip

It is worth to mention that in our proofs we may avoid the assumption on the potential $V\in C^\beta(\mathbb{R}^N)$, (see (V2) in \cite[Section 2]{Frank-Lenzmann-Silvestre}), that allows us to get the results also in the case $-2s<\alpha<0$. 
Moreover we can also avoid the homotopy argument in the potential $V$ (see \cite[Section 6]{Frank-Lenzmann-Silvestre}). 
Another, and maybe the main one, difference between our strategy and the one developed in \cite{Frank-Lenzmann, Frank-Lenzmann-Silvestre} is that the qualitative properties of the branch of solutions are obtained in a different and more direct way. 
In fact, the a priori bounds for the branch of solutions are obtained here by the Gidas-Spruck blow-up method. Remarkable is that the original problem does not have a priori bounds. 
The second main difference is that the positive preserving properties of the branch are obtained directly from the equation without using semigroups type arguments as occurs in \cite{Frank-Lenzmann} and \cite{Frank-Lenzmann-Silvestre}.
{We conclude this introduction mentioning that other uniqueness results for nonlocal equations can be found in  \cite{del Mar} and the references therein.}

The rest of the paper is organized as follows: In Section 2 we study the singular and non-singular eigenvalue and we get the proof our non-degeneracy result. 
Section 3 deals with the uniqueness issue. 
We remark here that through this paper we will denote by $C$ to a positive constant that may change from line to line and $B_R$ will be the ball of radius $R$ centered at the origin. 

\section{Non-degeneracy}

\subsection{Some Preliminaries}

We start by recalling the following result in \cite{Barrios-Quaas}, where the precise notion of solutions and the regularity 
issue between \eqref{transformado} and \eqref{P} were studied.

\begin{theorem} If
\begin{enumerate}
\item[i)] $\frac{1}{2}\leq s<1$ and $\alpha >-2s$, or
\item[ii)] $0<s<\frac{1}{2}$ and  $-2s<\alpha< \frac{2s(N-1)}{1-2s}$, 
\end{enumerate}
then the problem \eqref{P} has a nonnegative variational radially symmetric solution $ {U}_{\alpha,s}$. 
Moreover, since $f(t) = t^{p^*_{\alpha,s}}$ is a H\"older function, the solution is classical and, in fact, is positive. 
Furthermore, all {{positive radially symmetric}} solutions $u$ of \eqref{P} are fast decay, that is, there exist constants $c_i>0$, $i=1,\,2$, such that 
\begin{equation}\label{fastdecay}
c_1 r^{-N+2s}\leq u(r)\leq c_2 r^{-N+2s}, {\quad r\ge 1.}
\end{equation}
\end{theorem} 


Let now begin with the study of eigenvalue problems. As was mentioned in the introduction in \cite[Theorem 2.3]{Frank-Lenzmann-Silvestre} was shown that the second radial eigenfunction of operators of the form {{$H=(-\Delta)^s+V$}} changes sign exactly once on the half-line $(0,+\infty)$. 
This result may be regarded as an analog of the classical oscillation bound for classical Sturm-Liouville problems. 
In particular, this optimal oscillation result generalizes the one in \cite{Frank-Lenzmann} to an arbitrary dimension $N > 1$. 
We highlight that  such an oscillation estimate is a central ingredient in the proof of the non-degeneracy of the {{positive radially symmetric solution}} $U_{\alpha,s}$ of the problem (\ref{P}).

Unless otherwise stated, and as long as there is no doubt of what type of operator we are referring to, throughout this section to simplify the notation, we will denote $ {\mathfrak{L}}= {\mathfrak{L}}_{U_{\alpha,s}}$ (resp. $\widehat{\mathfrak{L}}=\widehat{\mathfrak{L}}_{\widehat{U}_{\alpha,s}}$)

In order to prove Theorems \ref{main}-\ref{nodegenerancia_extendido}{\color{black}{, by}} following the ideas of \cite {Frank-Lenzmann, Frank-Lenzmann-Silvestre} we have to do a carefully analysis of the eigenvalues of the linearized operators $\mathfrak{L}$ and $\widehat{\mathfrak{L}}$. Moreover, as we will see, the relationship between the radial eigenfunctions of both operators will be fundamental to conclude the non-degeneracy result \textcolor{black}{(see Lemma \ref{lema4}).}

\subsection{First eigenvalue for the original singular problem}
 

\textcolor{black}{Let us consider $\mathfrak{L}$ defined on the space
$H^s_{\mathrm{rad}}(\mathbb{R}^N)$ endowed with the norm}  
$$
\|u\|^2_{H^s_{\mathrm{rad}}(\mathbb{R}^N )} ={{C_{N,s}}}\iint_{\mathbb{R}^N\times\mathbb{R}^N} \frac{|u(x)-u(y)|^2}{|x-y|^{N+2s}}dx\,dy+\|u\|^2_{L^2(\mathbb{R}^N)}.
$$
We notice that
$$ 
H^s_{\mathrm{rad}}(\mathbb{R}^N )\hookrightarrow L^{p}(\mathbb{R}^N),
$$
with continuous embedding for $1\leq p\leq {{2^*_{s}:=\frac{2N}{N-2s}}}$,\,  $N>2s$, see \cite[Theorem 6.5]{DiNezza-Palatucci-Valdinoci}. 
Moreover, since, by \eqref{fastdecay},  $|U_{\alpha,s}|^{p^*_{\alpha,s}-1}$ is bounded and decays at infinity as $|x|^{-4s-2\alpha}$, by using that $\alpha>-2s$, one gets
\begin{equation}\label{iemb}
\int_{\mathbb{R}^N} |x|^{\alpha} |U_{\alpha,s}|^{p^*_{\alpha,s}-1}\phi^2dx \leq C \int_{\mathbb{R}^N}  
  \frac{\phi^2}{|x|^{2s}}dx <+\infty\quad \mbox{for all }\phi\in H^s_{\mathrm{rad}}(\mathbb{R}^N),
\end{equation}
and, when $\alpha\geq 0$, 
\begin{equation}\label{iemb-2}
\int_{\mathbb{R}^N} |x|^{\alpha} |U_{\alpha,s}|^{p^*_{\alpha,s}-1}\phi^2dx \leq C \int_{\mathbb{R}^N}  
 {\phi^2}dx <+\infty\quad \mbox{for all }\phi\in H^{s}_{\mathrm{rad}} (\mathbb{R}^N).
\end{equation}
%
We consider now the Rayleigh functional associated to $\mathfrak{L}$, given by
\begin{eqnarray}
\mathcal{R}(v)\!\!&\!:=\!\!&\int_{\mathbb{R}^N}|(-\Delta)^{\frac{s}{2}}v|^2 -p^*_{\alpha,s} |x|^{\alpha} |U_{\alpha,s}|^{p^*_{\alpha,s}-1} v^2\,dx\label{Erre}
\\
\!\!&=\!\!&{{C_{N,s}}}\iint_{\mathbb{R}^{2N}}\frac{(v(x)-v(y))^2}{|x-y|^{N+2s}}\, dx\, dy-p^*_{\alpha,s}\int_{\mathbb{R}^N}|x|^{\alpha} |U_{\alpha,s}|^{p^*_{\alpha,s}-1} v^2\,dx\nonumber\\
\!\!&=\!\!&
\langle v,\, \mathfrak{L}v\rangle \nonumber,\quad v\in H^{s}_{\mathrm{rad}} (\mathbb{R}^N),
\end{eqnarray}
and set
\begin{equation}\label{la1}
\lambda_1:=\inf_{\scriptsize
{\begin{array}[c]{cc} 
v \in  H^s_{\mathrm{rad}}(\mathbb{R}^N)\setminus\{0\} 
\end{array}}}
\frac{\mathcal{R}(v)}{\|v\|^2_{L^2(\mathbb{R}^N,|x|^{-2s})}}.
\end{equation}
Notice that  $\lambda_1>-\infty$ due to (\ref{iemb}) and the Hardy-Sobolev inequality
\begin{equation}\label{Hardy}
\mathcal{A}_{N,s}\int_{\mathbb{R}^N}{\frac{u^2}{|x|^{2s}}\,dx}\leq \|(-\Delta)^{\frac{s}{2}}u\|^2_{L^2(\mathbb{R}^{N})},\quad u\in H^s_{\mathrm{rad}}(\mathbb{R}^N),
\end{equation}
where
\begin{equation}\label{lambda_de_hardy}
\mathcal{A}_{N,s}=2^{2s}\frac{\Gamma^2\left(\frac{N+2s}{4}\right)}{\Gamma^2\left(\frac{N-2s}{4}\right)},
\end{equation} 
(see \textcolor{black}{\cite[Remark 2.1]{Manuel}}, \cite[Proposition 2.7]{DelaTorre-DelMar} \textcolor{black}{and \cite[Remark 4.7 iii)]{Barrios-Quaas}}), is an optimal constant that cannot be achieved (see \cite{beckner, Frank-Seiringer, Herbst}). 
We obtain the next result.
\begin{proposition}\label{saleg-1}
The following statements hold
\begin{enumerate}
\item[{\em i})] $\lambda_1<0$.\smallskip
\item[{\em ii})] Every minimizing sequence of (\ref{la1}) has a subsequence which weakly converges in $H^{s}_{\mathrm{rad}}(\mathbb{R}^N)$, and strongly in $L^r(B_R))$ for $1\leq r<2^*_s$ and every $R>0$.\smallskip
\item[{\em iii})] There exists a unique positive minimizer ${\varphi}_{1}$ of (\ref{la1}). 
Moreover, $\lambda_1$ is a simple eigenvalue of $\mathfrak{L}$ and ${\varphi}_1$ is an eigenfunction associated to $\lambda_1$.
\end{enumerate}
\end{proposition}
\begin{proof} 
We prove each part separately. 
\begin{enumerate}
\item[{\em i)}]  From \eqref{Ele} and the fact that $U_{\alpha,s}$ is a solution of \eqref{P}, it is clear {{that
$$
\mathfrak{L} (U_{\alpha,s})=  (1- p^*_{\alpha,s})|x|^{\alpha}U_{\alpha,s}^{p^*_{\alpha,s}}<0\quad \mbox{in }\mathbb{R}^N\setminus\{0\}.
$$
Hence,}} $\langle U_{\alpha,s}, \mathfrak{L}(U_{\alpha,s})\rangle<0$, so that $\lambda_1<0$.\smallskip
\item[{\em ii)}] Consider now a sequence $\{\phi_n\}\subseteq H^s_{\mathrm{rad}}(\mathbb{R}^N)$ with $\|\phi_n\|_{L^2(\mathbb{R}^N)}=1$, which minimizes (\ref{la1}), that is
\begin{equation}\label{dos}
\lim_{n\to\infty}\frac{\mathcal{R}(\phi_n)}{\|\phi_n\|^2_{L^2(\mathbb{R}^N,|x|^{-2s})}}=\lambda_1,
\end{equation}
with $\mathcal{R}$ given in \eqref{Erre}. 
We notice that $\phi_n$ is bounded in $H^s_{\mathrm{rad}}(\mathbb{R}^N)$. 
Indeed, given $\varepsilon>0$, fixed but arbitrary small, there exists $n_0$ such that
\begin{equation}\label{paratres}
\int_{\mathbb{R}^N}|(-\Delta)^{\frac{s}{2}} \phi_n|^2 \,dx\leq(\varepsilon+\lambda_1)\int_{\mathbb{R}^{N}}\frac{\phi_n^2}{|x|^{2s}}+p^*_{\alpha,s} \int_{\mathbb{R}^{N}}|x|^{\alpha} U_{\alpha,s}^{p^*_{\alpha,s}-1}\phi_n^2\, dx,\quad n\geq n_0.
\end{equation}
Thus, if $\alpha \geq 0$, by using (\ref{iemb-2}) and the fact that $\lambda_1<0$ there exists $n_0$ sufficiently large such that,
$$
\int_{\mathbb{R}^N}|(-\Delta)^{\frac{s}{2}} \phi_n|^2 \,dx \leq C\neq C(n),
$$
for every $n\geq n_0$. To obtain the same uniform bound for the case $-2s<\alpha<0$ is more involved. 
Indeed, let us consider $\{D_k\}_{k\in \mathbb{N}}$ an increasing sequence with $D_k>1$ for all $k\in\mathbb{N}$. 
Since $\lambda_1<0$ and $U_{\alpha,s}$ is radially decreasing in $|x|$, with $U_{\alpha,s}(0)$ bounded, choosing $k$ sufficiently large, by \eqref{paratres} and the fact that $|x|^{\alpha} U_{\alpha,s}^{p^*_{\alpha,s}-1}\rightarrow \infty$ when $|x|\rightarrow 0$ there exists $a_k>0$ such that  
$$
\int_{\mathbb{R}^N}|(-\Delta)^{\frac{s}{2}} \phi_n|^2 \,dx 
\begin{array}[t]{lll}
\displaystyle <\int_{\mathbb{R}^N}p^*_{\alpha,s} |x|^{\alpha} U_{\alpha,s}^{p^*_{\alpha,s}-1}\phi_n^2dx\medskip\\
\displaystyle \leq \int_{\mathbb{R}^N}( p^*_{\alpha,s} |x|^{\alpha} U_{\alpha,s}^{p^*_{\alpha,s}-1}-D_k)_+ \phi_n^2dx + \int_{\mathbb{R}^N} D_k  \phi_n^2dx\medskip\\
\displaystyle = \int_{|x|<a_k} ( p^*_{\alpha,s} |x|^{\alpha} U_{\alpha,s}^{p^*_{\alpha,s}-1}-D_k) \phi_n^2dx + \int_{\mathbb{R}^N} D_k\phi_n^2dx\medskip\\
\displaystyle\leq   \|  \,p^*_{\alpha,s} |x|^{\alpha} U_{\alpha,s}^{p^*_{\alpha,s}-1}-D_k\|_{L^{\frac{N}{2s}}(B_{a_k})}\|\phi_n\|^2_{L^{2^*_{s}}(\mathbb{R}^N)}+ D_k\medskip\\
\displaystyle\leq  S_{N,s} \|  \,p^*_{\alpha,s} |x|^{\alpha} U_{\alpha,s}^{p^*_{\alpha,s}-1}-D_k\|_{L^{\frac{N}{2s}}(B_{a_k})}\int_{\mathbb{R}^N}|(-\Delta)^{\frac{s}{2}} \phi_n|^2 \,dx+ D_k,
\end{array}
$$
where $S_{N,s}$ is the sharp constant of the fractional Sobolev embedding. We observe now that, since $D_k\rightarrow \infty$, as $k\rightarrow \infty$, implies that $a_{k}\rightarrow 0$, as $k\rightarrow \infty$, we have
\begin{eqnarray*}
0< \|  \,p^*_{\alpha,s} |x|^{\alpha} U_{\alpha,s}^{p^*_{\alpha,s}-1}-D_{k+1}\|_{L^{\frac{N}{2s}}(B_{a_{k+1}})}&\!\!\leq\!\!& \|  \,p^*_{\alpha,s} |x|^{\alpha} U_{\alpha,s}^{p^*_{\alpha,s}-1}-D_k\|_{L^{\frac{N}{2s}}(B_{a_{k}})}\\
&\!\!<\!\!&\|  \,p^*_{\alpha,s} |x|^{\alpha} U_{\alpha,s}^{p^*_{\alpha,s}-1}\|_{L^{\frac{N}{2s}}(\mathbb{R}^N)}<\infty,
\end{eqnarray*}
if $D_{k+1}>D_k$. 
Hence, we can fix $k_0$ sufficiently large in order to obtain
$$
S_{N,s} \|  \,p^*_{\alpha,s} |x|^{\alpha} U_{\alpha,s}^{p^*_{\alpha,s}-1}-D_{k_0}\|_{L^{\frac{N}{2s}}(B_{a_{k_0}})} \leq \frac{1}{2}, 
$$
so that
$$
\int_{\mathbb{R}^N}|(-\Delta)^{\frac{s}{2}} \phi_n|^2 \,dx\leq 2D_{k_0}<\infty\quad \mbox{for every }n\in\mathbb{N}.
$$
Therefore, up to a subsequence, 
\begin{eqnarray}
\displaystyle \phi_{n}&\rightharpoonup& \phi \qquad \mbox{ weakly in } H^{s}_{\mathrm{rad}}(\mathbb{R}^N), \nonumber\\
\displaystyle \phi_{n}&\to& \phi \qquad \mbox{ strongly in } L^{r}(B_R), \quad 1\leq r < 2^*_{s} \label{estrella1},\\
\displaystyle \phi_{n}&\to& \phi \qquad \mbox{ a.e. in } B_R \mbox{ for every $R>0$}.\nonumber
\end{eqnarray}
We will prove now that $\phi$, is, in fact, a minimizer of \eqref{la1}. For that 
%
we consider first the case $\alpha\geq 0$. Let fix $\eta>0$. 
By \eqref{estrella1} and the fast decay property of $U_{\alpha,s}$ given in \eqref{fastdecay}, there exists $\overline{R}>0$ and $n_0\in\mathbb{N}$ such that 
\begin{eqnarray}
\left| \int_{\mathbb{R}^N} V(x)(\phi_n^2-\phi^2)\,dx \right| \!&\!\!\leq\!\!&
\!\displaystyle \int_{B_R}|V(x)|\,|\phi_n^2-\phi^2|\,dx +\int_{\mathbb{R}^N\setminus B_R} |V(x)||\phi_n^2-\phi^2|\,dx\nonumber\\
\!&\!\!\leq\!\!&\!C\left(\int_{ B_R}   |\phi_n^2-\phi^2|   \,dx +\frac{1}{R^{4s+\alpha}}\int_{\mathbb{R}^N\setminus B_R} |\phi_n^2-\phi^2|\,dx \right)\nonumber\\
\!&\!\!\leq\!\!&\!\eta,\label{eta}
\end{eqnarray}
if $R\geq \overline{R}$, $n\geq n_0$ where $V$ was given in \eqref{Uve}. Again, when $-2s<\alpha<0$ the computation is more tricky. 
In that case we can assume, without loss of generality that there exists $\xi>0$ such that $\alpha=-2s+\xi<0$. 
Then for every $2<\ell\leq 2^{*}_{s}$ we get
\begin{equation}\label{bajon}
\left|\int_{\mathbb{R}^{N}}{V(x)(\phi_n^2-\phi^2)\,dx}\right|\leq\left(\int_{\mathbb{R}^{N}}|V(x)|^{\frac{\ell}{\ell-2}}\,dx\right)^{\frac{\ell-2}{\ell}}\left(\int_{\mathbb{R}^{N}}(\phi_n^2-\phi^2)^{\frac{\ell}{2}}\,dx\right)^{\frac{2}{\ell}}.
\end{equation}
Since, by using the boundedness of $U_{\alpha,s}$ and \eqref{fastdecay}, by taking 
$$
2<\ell=\frac{2N}{N-2s+\frac{\xi}{2}}<2^*_s,\quad q=\frac{\ell}{\ell-2}=\frac{2N}{4s-\xi},
$$
it follows that
\begin{eqnarray*}
\int_{\mathbb{R}^{N}}|V(x)|^{\frac{\ell}{\ell-2}}\,dx&\leq& C\left(\int_{0}^{1} r^{N-1+\alpha\frac{\ell}{\ell-2}}\,dr+\int_{1}^{\infty} r^{N-1-\frac{(4s+\alpha)\ell}{\ell-2}}\,dr\right)\leq C,
\end{eqnarray*}
by \eqref{estrella1} and \eqref{bajon} we get that \eqref{eta} is also true for the case $-2s<\alpha< 0$.  Thus,
\begin{equation}\label{objetivo}
\int_{\mathbb{R}^N}|x|^{\alpha} |U_{\alpha,s}|^{p^*_{\alpha,s}-1}\phi_n^2\,dx\rightarrow \int_{\mathbb{R}^N}|x|^{\alpha} |U_{\alpha,s}|^{p^*_{\alpha,s}-1}\phi^2\,dx,
\end{equation}
for every $-2s<\alpha$. 
Therefore, using the lower semicontinuity of the norm, \eqref{dos} and \eqref{objetivo} on one hand we get
\begin{equation}\label{bajon1}
\mathcal{R}(\phi)=\mathcal{R}(\liminf_{n\to\infty}\phi_n)\leq \liminf_{n\to\infty}\mathcal{R}(\phi_n)=\lambda_1\liminf_{n\to\infty}\|\phi_n\|^{2}_{L^{2}(\mathbb{R}^{N}, |x|^{-2s})},
\end{equation}
where $\mathcal{R}$ was given in \eqref{Erre}. 
On the other hand, by definition of the infimum, one has
\begin{equation}\label{bajon2}
\mathcal{R}(\liminf_{n\to\infty}\phi_n)=\mathcal{R}(\phi)\geq \lambda_1\|\phi\|^{2}_{L^{2}(\mathbb{R}^{N}, |x|^{-2s})}.
\end{equation}
Therefore by \eqref{bajon1}-\eqref{bajon2}, {and the fact that $\lambda_1<0$}, we obtain
\begin{equation}\label{flow3}
\frac{\mathcal{R}(\phi)}{\|\phi\|^2_{L^{2}(\mathbb{R}^{N},|x|^{-2s})}}=\lambda_1.
\end{equation}
That is, $\phi$, usually denoted by $\varphi_1$, is a minimizer of \eqref{la1} as wanted. 
\item[{\em iii)}] We will prove now that $\lambda_1$ is simple, namely, that $\varphi_1$ is the unique minimizer of \eqref{la1} up to a constant. 
First of all we notice that, by the minimality, 
\begin{equation}\label{helping}
\left.\frac{J(\varphi_1+\eta\Phi)}{d\eta}\right|_{\eta=0}=\lim_{\eta\to 0}\frac{J(\varphi_1+\eta\Phi)-J(\varphi_1)}{\eta}=0,
\end{equation}
for every $\Phi\in H^s_{\mathrm{rad}}(\mathbb{R}^N)$ where
$$
J(v):=\frac{\mathcal{R}(v)}{\|v\|^2_{L^{2}(\mathbb{R}^{N},|x|^{-2s})}},\quad v\in H^s_{\mathrm{rad}}(\mathbb{R}^N),
$$
with $\mathcal{R}$ given in \eqref{Erre}. 
Then, by \eqref{flow3} and \eqref{helping}, it follows that
\begin{eqnarray}
0\!\!&\!\!=\!\!&\!\! {{C_{N,s}}} \iint_{\mathbb{R}^{2N}}\frac{(\varphi_1(x)-\varphi_1(y))(\Phi(x)-\Phi(y))}{|x-y|^{N+2s}}\, dx\, dy-p^*_{\alpha,s}\int_{\mathbb{R}^{N}}|x|^{\alpha}|U_{\alpha,s}|^{p^*_{\alpha,s}-1}\varphi_1\Phi\, dx\nonumber\\
&&-\lambda_1\int_{\mathbb{R}^{N}}\frac{\varphi_1\Phi}{|x|^{2s}}\, dx,\quad \Phi\in H^s_{\mathrm{rad}}(\mathbb{R}^N).\label{flow}
\end{eqnarray}
Thus, {{in a week sense
\begin{equation}\label{eqphistar}
(-\Delta)^s\varphi_1-p^*_{\alpha,s}|x|^{\alpha} |U_{\alpha,s}|^{p^*_{\alpha,s}-1}\varphi_1=\frac{\lambda_1}{|x|^{2s}}\varphi_1.
\end{equation}
Moreover,}} since
$$
\left||\varphi_1(x)|-|\varphi_1(y)|\right|\leq |\varphi_1(x)-\varphi_1(y)|,\quad x,y\in\mathbb{R}^N,
$$
then $\mathcal{R}(|\varphi_1|)\leq \mathcal{R}(\varphi_1)$, so we could assume, without loss of generality, that $\varphi_1\geq 0$. 
Observe that, again without loss of generality, we can also consider $\|\varphi_1\|_{L^2(\mathbb{R}^{N})}=1$. Let us define now $\varphi\in H^s_{\mathrm{rad}}(\mathbb{R}^N)$ another minimizer of \eqref{la1} and
$$
\bar\varphi(x):=\frac{\varphi(x)}{\|\varphi\|_{L^{2}(\mathbb{R}^{N})}},\quad x\in\mathbb{R}^{N}.
$$
By the previous arguments, 
\begin{eqnarray}
&&C_{N,s}\iint_{\mathbb{R}^{2N}}\frac{(\bar\varphi(x)-\bar\varphi(y))(\Phi(x)-\Phi(y))}{|x-y|^{N+2s}}\, dx\, dy-p^*_{\alpha,s}\int_{\mathbb{R}^{N}}|x|^{\alpha}|U_{\alpha,s}|^{p^*_{\alpha,s}-1}\bar\varphi\Phi\,dx\nonumber\\
&&\,\,=\lambda_1\int_{\mathbb{R}^{N}}\frac{\bar\varphi\Phi}{|x|^{2s}}\, dx,\label{flow1}
\end{eqnarray}
for every $\Phi\in H^s_{\mathrm{rad}}(\mathbb{R}^N)$. Moreover we could consider that $\bar\varphi\geq0$ a.e. in $\mathbb{R}^N$. By \eqref{flow} and \eqref{flow1}, it is clear that
$$
\frac{\mathcal{R}(\varphi_1-\bar\varphi)}{\|\varphi_1-\bar\varphi\|_{L^{2}(\mathbb{R}^{N},|x|^{-2s})}}=\lambda_1,
$$
that is, $\Psi(x):=\varphi_1(x)-\bar\varphi(x)$, $x\in\mathbb{R}^{N}$ is also a minimizer associated to $\lambda_1$. 
Thus we could assume $\Psi\geq0$ a.e. in $\mathbb{R}^{N}$, so that $\varphi_1(x)\geq\bar\varphi(x)$, a.e. $x\in\mathbb{R}^{N}$ which also gives
\begin{equation}\label{positivity}
\varphi_1^2(x)\geq\bar\varphi^2(x),\mbox{ a.e. $x\in\mathbb{R}^{N}$}.
\end{equation}
Since on the other hand
$$
\int_{\mathbb{R}^{N}}\big(\varphi_1^2(x)-\bar\varphi^2(x)\big)\, dx=0,
$$
then, by \eqref{positivity} we conclude that $\varphi_1=\bar\varphi$ a.e. in $\mathbb{R}^{N}$ as wanted.
\end{enumerate}
\end{proof}

The same result holds if we also minimize the quadratic form that appears in \eqref{la1} with some orthogonality  conditions. 
To this end we say that $\phi$ and $\varphi$ are orthogonal if they satisfy 
\begin{equation}\label{perp}
\phi\perp\varphi=\int_{\mathbb{R}^N} \frac{\phi\varphi}{|x|^{2s}} dx = 0,
\end{equation}
and we consider
\begin{equation}\label{la2}
\lambda_2:=\inf_{\scriptsize
{\begin{array}[c]{cc} 
v \in  H^s_{\mathrm{rad}}(\mathbb{R}^N)\setminus\{0\} \\
v\perp \mathrm{span}\{\varphi_1\} 
\end{array}}}
\frac{\mathcal{R}(v)}{\|v\|^2_{L^2(\mathbb{R}^N,|x|^{-2s})}},
\end{equation}
where $\varphi_1$ was given Proposition \ref{saleg-1}. 
Then we have the following
\begin{proposition}\label{propovalprop2}
Assume that $\lambda_2<\mathcal{A}_{N,s}$, with $\mathcal{A}_{N,s}$ given by \eqref{lambda_de_hardy}. 
Then:
\begin{enumerate}
\item[{\em i})] 
Every minimizing sequence of (\ref{la2}) has a subsequence which weakly converges in $H^{s}_{\mathrm{rad}}(\mathbb{R}^N)$, and strongly in $L^r(B_R))$ for $1\leq r<2^*_s$ for every $R>0$.\smallskip
\item[{\em ii})] There exists a minimizer ${\varphi}_{2}$ of (\ref{la2}). 
Moreover, $\lambda_2$ is an eigenvalue of $\mathfrak{L}$ and ${\varphi}_2$ is an eigenfunction associated to $\lambda_2$.
\end{enumerate}
\end{proposition}
\begin{proof}
Let us consider a minimizing sequence $\phi_n \in{ {H^s_{\mathrm{rad}}(\mathbb{R}^N)}}$, {{$\phi_n\perp \mathrm{span}\{\varphi_1\}$}}, such that $\|\phi_n \|_{L^2(\mathbb{R}^{N})}= 1$, which minimizes (\ref{la2}). 
Note that $\phi_n$ is bounded in {{$H^s_{\mathrm{rad}}(\mathbb{R}^N)$}} for every $\alpha>-2s$. 
Indeed, doing similar computations as in \eqref{paratres} we get that, for every, fixed but arbitrarily small $\varepsilon>0$,
$$
\left(1-\frac{\lambda_2+\varepsilon}{\mathcal{A}_{N,s}}\right)\int_{\mathbb{R}^N}|(-\Delta)^{\frac{s}{2}} \phi_n|^2 \,dx\leq p^*_{\alpha,s} \int_{\mathbb{R}^{N}}|x|^{\alpha} U_{\alpha,s}^{p^*_{\alpha,s}-1}\phi_n^2\, dx,
$$
if $n\geq n_0(\varepsilon)$. 
Thus, by \eqref{iemb-2} we conclude that there exists $n_0$ sufficiently large such  that 
$$
\int_{\mathbb{R}^N}|(-\Delta)^{\frac{s}{2}} \phi_n|^2 \,dx  <\infty,\quad n\geq n_0.
$$
Thus the minimizing sequence $\phi_n$ converges, up to a subsequence, to a function $\phi$ weakly in $H^s_{\mathrm{rad}}(\mathbb{R}^N)$, strongly in $L^2(B_R)$ for all $R>0$ and pointwise a.e. Since, subsequently, $\phi_n$ also converges weakly in $L^{2}(\mathbb{R}^N,|x|^{-2s})$,
$$
0 = \lim_{k\rightarrow \infty} \int_{\mathbb{R}^N} |x|^{-2s} \phi_{n_k} \varphi_1\,dx =\int_{\mathbb{R}^N} |x|^{-2s}\phi \varphi_{1}\,dx\quad \mbox{as } k\rightarrow \infty, 
$$
obtaining that {{$\phi\perp \mathrm{span}\{\varphi_1\}$}}. 
Then, doing as in Proposition \ref{saleg-1} ii), it follows that $\lambda_2$ is attained. 
In fact we can assume, without loss of generality, that $\lambda_2=\mathcal{A}_{N,s}-\delta$ for some $\delta>0$. 
Then it is clear that
$$
0>\lambda_2-\left(\mathcal{A}_{N,s}-\frac{\delta}{2}\right)=\lim_{n\rightarrow \infty}\frac{\mathcal{R}^*(\phi_n)}{\|v_n\|^2_{L^2(\mathbb{R}^N,|x|^{-2s})}},
$$
where {{
$$
\mathcal{R}^*(\phi_n):=\|\phi_n\|^{2}_{*}+\int_{\mathbb{R}^N}V(x)\phi_n^2\,dx,
$$
with}} $V$ given in \eqref{Uve} and 
$$
\|v\|^2_{*}:=\|(-\Delta)^{\frac{s}{2}}v\|^{2}_{L^{2}(\mathbb{R}^{N}}-\left(\mathcal{A}_{N,s}-\frac{\delta}{2}\right)\|v\|^2_{L^2(\mathbb{R}^N,|x|^{-2s})},\, v\in H^{s}(\mathbb{R}^{N}).
$$
Since by \eqref{Hardy}, $\|v\|^2_{*}$ is equivalent to $\|(-\Delta)^{\frac{s}{2}}v\|^{2}_{L^{2}(\mathbb{R}^{N})}$, by the lower semicontinuity and \eqref{objetivo} on one hand it follows that
$$
\mathcal{R}^*(\phi)\leq\liminf_{n\to\infty}\mathcal{R}^*(\phi_n)=\left(\lambda_2-\left(\mathcal{A}_{N,s}-\frac{\delta}{2}\right)\right)\liminf_{n\to\infty}\|\phi_n\|^{2}_{L^2(\mathbb{R}^N,|x|^{-2s})}.
$$
On the other hand, by definition of the infimum, we also have
$$
\mathcal{R}^*(\phi)\geq \left(\lambda_2-\left(\mathcal{A}_{N,s}-\frac{\delta}{2}\right)\right)\|\phi\|^{2}_{L^2(\mathbb{R}^N,|x|^{-2s})},
$$
we conclude by using the Fatou's Lemma.
\end{proof}
The previous result can be generalized. In fact, if we consider \textcolor{black}
{$n\geq 1$ an arbitrary natural number and} $M$ a $(n-1)$-dimensional subspace spanned by the eigenfunctions corresponding to the eigenvalues $\lambda_1, \dots, \lambda_{n-1}$, {\it i.e.} {{$ M=\mathrm{span}\{\varphi_1,\ldots ,\varphi_{n-1}\} $}}, and we define
\begin{equation}\label{la-i}
\lambda_n:=\inf_{\scriptsize
{\begin{array}[c]{cc} 
v \in  H^s_{\mathrm{rad}}(\mathbb{R}^N)\setminus\{0\} \\
v\perp M 
\end{array}}}
\frac{\mathcal{R}(v)}{\|v\|^2_{L^2(\mathbb{R}^N,|x|^{-2s})}},
\end{equation}
following the ideas used in the proof of the  previous proposition, we get the next.
\begin{proposition} \label{propola-i}
Assume that $\textcolor{black}{\lambda_n}<\mathcal{A}_{N,s}$. 
Then:
\begin{enumerate}
\item[{\em i})] Every minimizing sequence of (\ref{la-i}) has a subsequence which weakly converges in $H^{s}_{\mathrm{rad}}(\mathbb{R}^N)$, and strongly in $L^r(B_R))$ for $1\leq r<2^*_s$ for every $R>0$.\smallskip\smallskip
\item[{\em ii})] There exists a minimizer $\textcolor{black}{\varphi}_{n}$ of (\ref{la-i}). 
Moreover, $\textcolor{black}{\lambda_n}$ is an eigenvalue of $\mathfrak{L}$ and $\textcolor{black}{\varphi}_n$ is an eigenfunction associated to $\textcolor{black}{\lambda_n}$.
\end{enumerate}
\end{proposition}

Observe that, by construction, and the simplicity of the first eigenfunction, we have
$$
\lambda_1<\lambda_2\leq \ldots \leq \lambda_n\leq \ldots
$$
Now we review some results that follow the ideas developed in \cite[Section 5]{Frank-Lenzmann-Silvestre} (see also \cite{Frank-Lenzmann}). 
These results, in particular, give us an upper bound for the number of sign-change of the second eigenfunction of the linearized operator $\mathfrak{L}$ that, as we commented at {{the beginning of this section, will}} be one of the key point to establish our non-degeneracy Theorem.

The main difference of the computations that we will show with respect to the ones done in \cite{Frank-Lenzmann-Silvestre}, is that here we consider singular eigenvalue. 
As we mentioned in the introduction of this work, the key tool to establish the fundamental results needed to proved the nondegeneracy Theorem, are based in a Courant type nodal domain Theorem for the extended problem introduced in \cite{Caf-Sil}. To work with it let us consider 

$$
a=1-2s\in (-1,1),\quad d_s=2^{2s-1} \frac{\gamma(s)}{\gamma(1-s)} > 0,\quad 0<s<1,
$$
and the energy functional 
$$
\mathcal{F}(u) := d_s    \iint_{\mathbb{R}_+^{N+1}} |\nabla u(x,t)|^2t^a\, dx\, dt +  \int_{\mathbb{R}^N} V (x)|u(x, 0)|^2 \,dx,
$$
defined for {{
$$
u\in \mathcal{H}^{1,a}(\mathbb{R}_+^{N+1}):=\{u\in \dot{\mathcal{H}}^{1,a}(\mathbb{R}_+^{N+1}) : u(x, 0)\in L^2(\mathbb{R}^N)\},
$$
where $\dot{\mathcal{H}}^{1,a}(\mathbb{R}_+^{N+1})$}} is the completion of $C_0^{\infty}(\mathbb{R}^{N+1}_+)$  with respect to the homogeneous Sobolev norm
$$ 
\|u\|^{2}_{\mathcal{H}^{1,a}(\mathbb{R}_+^{N+1})}= \iint_{\mathbb{R}_+^{N+1}}  |\nabla u|^2 t^a \,dx \,dt,
$$
and $V$ {{is }} given in \eqref{Uve}. Since it can be seen that there exists a well-defined trace operator $T : \mathcal{H}^{1,a}(\mathbb{R}_+^{N+1})\rightarrow \mathcal{H}^s(\mathbb{R}^N)$, by $u(x,0)$ we denote the trace of $u(x,t)$ on $\partial \mathbb{R}_+^{N+1}$, that is, as usual $u(x,0):=(Tu)(x)$. It is well-known (see for instance \cite{BCdPS}) that
\begin{equation}\label{FL}
\iint_{\mathbb{R}^{N+1}} |\nabla u|^2t^a \,dx\, dt  \geq \frac{1}{d_s} \int_{\mathbb{R}^{N}}  |(-\Delta)^{\frac{s}{2}} Tu|^2 dx,
\end{equation}
where $\mbox{the equality holds if and only if }u = \mathcal{E}_sf\mbox{ for some }f\in {H}^s(\mathbb{R}^N)$, with
%
$$
\mathcal{E}_sf(x,t):= (\mathcal{P}_s(\cdot,t)\ast f)(x). 
$$
Here 
$$
\mathcal{P}_s(x,t)=\frac{1}{x}P_s\left(\frac{x}{t}\right)=c_s\frac{1}{x}\left(\frac{t^2}{t^2+x^2}\right)^{\frac{2-a}{2}},\, x\in\mathbb{R},\, t>0,
$$
denotes the $s$-Poisson extension of $f : \mathbb{R}^N\rightarrow\mathbb{R}$ to the upper halfspace $\mathbb{R}_+^{N+1}$.  

We also want to highlight that, since
$$
V\in
\begin{cases}
L^{p}(\mathbb{R}^N),\mbox{ for every $p>\frac{N}{2s}$} &\text{if }\alpha>0,\\[8pt]
L^{\frac{2N}{2s-\alpha}}(\mathbb{R}^N)&\text{if }-2s<\alpha<0,
\end{cases}
$$
then $V$ is in the so called Kato class, see \cite{Frank-Lenzmann-Silvestre}.

Finally we ant to mention that, since we are interested in ${H} = (-\Delta)^s+V$ with $V$ with a {{radially symmetric}} potential, will be natural to work with the closed subspace
$$
\mathcal{H}^{1,a}_{\mathrm{rad}}(\mathbb{R}_+^{N+1}):=  \{u\in \mathcal{H}^{1,a}(\mathbb{R}_+^{N+1}) : u(x,t)\mbox{ is radial in }x\in\mathbb{R}^N\mbox{ for a.e. }t > 0 \}.
$$
We can show now the auxiliary fundamental results. The first one is the following. 
\begin{proposition} \label{Propo1}
Let $N\geq 1$, $0 < s < 1$ and $V$ given in \eqref{Uve}. Assume that 
$$
H = (-\Delta)^s {+} V,
$$
acting on $L^2_{\mathrm{rad}}(\mathbb{R}^N,|x|^{-2s})$ has at least $n$ singular eigenvalues
$$
\lambda_1\leq \lambda_2 \leq  \ldots \leq \lambda_n < \mathcal{A}_{N,s},
$$
where $\mathcal{A}_{N,s}$ was given in \eqref{lambda_de_hardy}.
Furthermore, let $M$ be an $(n-1)$-dimensional subspace of $L^2_{\mathrm{rad}}(\mathbb{R}^N,|x|^{-2s})$ spanned by the eigenfunctions corresponding to the eigenvalues $\lambda_1, \dots, \lambda_{n-1}$. 
Then we have 
$$
\lambda_n= \inf\bigg\{\mathcal{F}(u) : u\in\mathcal{H}_{\mathrm{rad}}^{1,a}(\mathbb{R}^{N+1}), \int_{\mathbb{R}^N} |x|^{-2s}|u(x,0)|^2 \,dx = 1, u(\cdot,0)\perp M \bigg\},
$$
where the orthogonality is  given by \eqref{perp}.

Moreover, the infimum is attained if and only if $u = \mathcal{E}_a f$ with $f\in H_{\mathrm{rad}}^s(\mathbb{R}^N )$, where $ \|f\|_{L^2(\mathbb{R}^{N}, |x|^{-2s})}^2= 1$ and $f\in M^{\perp} $ is a linear combination of eigenfunctions of $H$ corresponding to the eigenvalue $\lambda_n$.
\end{proposition}
\begin{proof} 
By \eqref{FL} we get that the infimum is bounded from bellow by
$$
\inf\bigg\{\int_{\mathbb{R}^{N}}\big(|(-\Delta)^{\frac{s}{2}}f|^2{+} V (x)|f(x)|^2\big) dx : f\in H_{\mathrm{rad}}^s(\mathbb{R}^N ), \int_{\mathbb{R}^N} |x|^{-2s}|f(x)|^2 \,dx = 1, f\perp M \bigg\},
$$
where the equality is attained if $u =  \mathcal{E}_af$. 
Thus we conclude by Proposition \ref{propola-i}.
\end{proof}

Notice now that, since, by standard regularity results and the fact that the normal derivative exists and is $(-\Delta)^s\psi$, any radial eigenfunction $\psi$  associated to the singular eigenvalue of $H = (-\Delta)^s {+} V$ is continuous in $\mathbb{R}^N\setminus\{0\}$, its extension $\mathcal{E}_a\psi$ belongs to $C^0(\overline{\mathbb{R}^{N+1}_+}\setminus\{0\})$, we can consider its nodal domains.
These nodal domains are defined as the connected components of the open set 
$$
\{(x,t)\in \mathbb{R}_+^{N+1} : (\mathcal{E}_a\psi)(x,t)\neq 0\},
$$
in the upper half-space $\mathbb{R}_+^{N+1}\setminus\{0\}$. By using Proposition \ref{Propo1} we proceed now to derive upper bounds on the number of nodal domains for extension of the eigenfunctions of $H$ to the upper half-space $\mathbb{R}_+^{N+1}$. 
That is, we get the following result that corresponds to \cite[Proposition 5.2]{Frank-Lenzmann-Silvestre} adapted to our setting.

\begin{proposition}  \label{prop53} 
Let $N\geq 1$, $0<s<1$ and $V$ given in \eqref{Uve}. Suppose that $H = (-\Delta)^s {+} V$, acting on $L^2_{\mathrm{rad}}(\mathbb{R}^N,|x|^{-2s})$, has at least $n$ eigenvalues
$$\lambda_1< \lambda_2 \leq \ldots \lambda_n <  \mathcal{A}_{N,s}.$$
If $\psi_n\in H^s_{\mathrm{rad}}(\mathbb{R}^N )$ is a real eigenfunction of $H$ with eigenvalue $\lambda_n$, then its extension $\mathcal{E}_a\psi_n$, with $a = 1-2s$, has at most $n$ nodal domains on $\mathbb{R}_+^{N+1}\setminus\{0\}$. 
\end{proposition}
\begin{proof} 
We follows verbatim \cite[Proposition 5.2]{Frank-Lenzmann-Silvestre} with $\gamma_i$ taked such that $\|u\|_ {L^2_{\mathrm{rad}}(\mathbb{R}^N,|x|^{-2s})}=1$  and with the orthogonality $u \perp M$ defined in \eqref{perp}.
\end{proof} 

We conclude this subsection giving an estimate of the number of the sign changes of the second eigenfunction of $H$. Before that we introduce the following useful definition.  
%
\begin{definition}
Let $\psi\in C^0(\mathbb{R}^N)$ be radial and let $m\geq 1$ be an integer. We say that $\psi(r)$ changes its sign $m$ times on $(0,+\infty)$, if there exist $0 < r_1 < \ldots< r_{m+1}$ such that $\psi(r_i)\neq  0$ for $i = 1,\ldots,m + 1$ and $\mathrm{sign}(\psi(r_i)) = -\mathrm{sign}(\psi(r_{i+1}))$ for $i = 1,\ldots,m$.
\end{definition}
%
We can state now the following result.

\begin{proposition}\label{clave}
Let $N\geq 1$, $0 < s < 1$ and $V$ given in \eqref{Uve}. Consider $H = (-\Delta)^s \textcolor{black}{+} V$ acting on $L^2_{\mathrm{rad}} (\mathbb{R}^N,|x|^{-2s})$. Let $\lambda_2 < \mathcal{A}_{N,s}$ be the second eigenvalue of $H$ acting on $L^2_{\mathrm{rad}}(\mathbb{R}^N )$ and $\varphi_2\in L^2_{\mathrm{rad}} (\mathbb{R}^N,|x|^{-2s})$ is a real-valued solution of $H\varphi_2 = \lambda_2\varphi_2$. Then $\varphi_2$ changes its sign at most twice on $(0, +\infty)$.
\end{proposition}
\begin{proof}
The proof follows the ideas developed in \cite[Proposition 5.3]{Frank-Lenzmann-Silvestre} by replacing \cite[Proposition 5.2]{Frank-Lenzmann-Silvestre} by Proposition \ref{prop53}. 
\end{proof}

\subsection{First eigenvalue for the transformed problem}

We study now the linearized limit problem by introducing the linear operator 

$$
\widehat{\mathfrak{L}}:   {\widehat{\mathcal{H}}^s}(\mathbb{R})\rightarrow L^{2}(\mathbb{R}),
$$

where $\widehat{\mathfrak{L}}=\widehat{\mathfrak{L}}_{\widehat{U}_{\alpha,s}}$ was defined in  \eqref{L_transformado} and
\begin{equation}\label{el_espacio}
\widehat{\mathcal{H}^s}(\mathbb{R}):=\left\{u:\mathbb{R}\to\mathbb{R}:\,\|u\|_{\widehat{\mathcal{H}^s}(\mathbb{R})}<\infty\right\},
\end{equation}
where 
$$
\|u\|^2_{\widehat{\mathcal{H}^s}(\mathbb{R})}:=C_{N,s}\int_{\mathbb{R}}\int_{\mathbb{R}}(u(\kappa)-u(\tau))^2\mathcal{K}(\kappa-\tau)\, d\tau\, d\kappa +\|u\|_{L^{2}(\mathbb{R})}^2,
$$
with $\mathcal{K}$ given in \eqref{kernel}. \textcolor{black}{Along all this section, since there is not any doubt about it, we will denote $\widehat{\mathcal{H}}(\mathbb{R})$ instead of $\widehat{\mathcal{H}^s}(\mathbb{R})$ in order to simplify the notation.}
By \cite[Proposition 4.2-4.3]{Barrios-Quaas} in the case $0<s<\frac{1}{2}$, we know that the space $\widehat{\mathcal{H}}(\mathbb{R})$ is continuously embedded in $L^q(\mathbb{R})$ for every $1\leq q\leq {{\frac{2}{1-2s}:=2^{**}_{s}}}$ and compactly if $1\leq q < 2_s^{**}$. If $s\geq\frac{1}{2}$ we highlight that there is no critical Sobolev exponent in dimension one. 
Defining now the transformed functional
\begin{equation}\label{la00}
\begin{array}[t]{lll}
\widehat{\mathcal{R}}(\varphi)\!\!&\!\!:= \displaystyle {{C_{N,s}}} \iint_{\mathbb{R}^2} (\varphi(\kappa)-\varphi(\tau))^2 \mathcal{K}(\kappa-\tau)\,d\tau\,d\kappa+\int_{\mathbb{R}}\left({\mathcal{A}_{s,N}}-p^*_{\alpha,s} |\widehat{U}_{\alpha,s}|^{p^*_{\alpha,s}-1}\right)\varphi^2(\kappa)\,d\kappa\medskip\\
\!\!&\!\!\,\,=\langle \varphi, \widehat{\mathfrak{L}}(\varphi)\rangle,\quad \varphi\in \widehat{\mathcal{H}}(\mathbb{R}).
\end{array}
\end{equation}
and the corresponding transformed first eigenvalue
\begin{equation}\label{la11}
\widehat{\lambda}_1:=\inf_{\widehat{v}\in \widehat{\mathcal{H}}(\mathbb{R})\setminus\{0\}} \frac{\widehat{\mathcal{R}}(\widehat{v})}{\|\widehat{v}\|_{L^2(\mathbb{R})}^{2}},
\end{equation}
proceeding similarly to the proof of Proposition \ref{saleg-1}, we get the next
\begin{proposition}\label{saleg-00}
The following statements hold:
\begin{enumerate}
\item[{\em i})] $\widehat{\lambda}_1<0$.\smallskip
\item[{\em ii})] Every minimizing sequence for (\ref{la11}) has a subsequence
which weakly converges in $\widehat{\mathcal{H}}(\mathbb{R})$, and strongly in $L^r(-R,R)$ for $1\leq r<2^{**}_s$ and every $R>0$.\smallskip 
\item[{\em iii})] There exists a unique positive minimizer $\widehat{\varphi}_{1}$ for (\ref{la11}). Moreover, $\widehat{\lambda}_1$ is an eigenvalue of $\widehat{\mathfrak{L}}$ and $\widehat{\varphi}_1$ is an eigenfunction associated to $\widehat{\lambda}_1$.
\end{enumerate}
\end{proposition}

\begin{remark}\label{auxi}
{{We notice that, by using a symmetric decreasing rearrangement, see \cite[Collorary 2.3]{lieb},  we can always assume that $\widehat{\varphi}_{1}$ is an even function.}}
\end{remark}
\begin{proof} 
As we did in Proposition \ref{saleg-1} we prove each part separately. 
\begin{enumerate}
\item[{\em i)}] Since
\begin{equation}\label{flower}
\widehat{\mathfrak{L}}(\widehat{U}_{\alpha,s})=  (1- p^*_{\alpha,s})\widehat{U}_{\alpha,s}^{p^*_{\alpha,s}}<0,
\end{equation}
we clearly get that $\widehat{\lambda}_1<0$.\smallskip
\item[{\em ii)}] Consider now a sequence $\{\phi_n\}$ in $\widehat{\mathcal{H}}(\mathbb{R})$ with $\|\phi_n\|_{L^2(\mathbb{R})}=1$, which minimizes (\ref{la11}). 
Thus, given $\varepsilon>0$, fixed but arbitrary small, there exists $n_0$ such that 
%
\begin{eqnarray*}
{{C_{N,s}}}\iint_{\mathbb{R}^2} (\phi_n(\kappa)-\phi_n(\tau))^2 \mathcal{K}(\kappa-\tau)\,d\tau\,d\kappa+\mathcal{A}_{s,N}&\leq&(\varepsilon+\widehat{\lambda}_1)+p^*_{\alpha,s}\int_{\mathbb{R}} |\widehat{U}_{\alpha,s}|^{p^*_{\alpha,s}-1}\phi_n^2\, d\kappa\\
&\leq& C\neq C(n),
\end{eqnarray*}
if  $n\geq n_0$ by using the fact that $\widehat{U}_{\alpha,s}$ is bounded and $\lambda_1<0$. 
Therefore, up to a subsequence, 
\begin{eqnarray*}
\displaystyle \phi_{n}&\rightharpoonup& \phi \qquad \mbox{ weakly in } \widehat{\mathcal{H}}(\mathbb{R}), \nonumber\\
\displaystyle \phi_{n}&\to& \phi \qquad \mbox{ strongly in } L^{r}(-R,R), \quad 1\leq r < 2^{**}_{s},\\
\displaystyle \phi_{n}&\to& \phi \qquad \mbox{ a.e. in } (-R,R),\mbox{ for every $R>0$}.\nonumber
\end{eqnarray*}
Since, doing in a similar way as in \eqref{eta}, we also get that
$$
\int_{\mathbb{R}} p^*_{\alpha,s} |\widehat{U}_{\alpha,s}|^{p^*_{\alpha,s}-1} \phi_n^2d\kappa \rightarrow \int_{\mathbb{R}} p^*_{\alpha,s} |\widehat{U}_{\alpha,s}|^{p^*_{\alpha,s}-1} \phi^2d\kappa,\quad n\rightarrow +\infty,
$$
by the hypothesis $\|\phi_n\|_{L^2(\mathbb{R})}=1$ and the lower semicontinuity, the previous limit implies
$$
\widehat{\lambda}_1\|\phi\|^2_{L^{2}(\mathbb{R})}\leq \widehat{\mathcal{R}}(\phi)=\widehat{\mathcal{R}}(\liminf_{n\to\infty}\phi_n)\leq \liminf_{n\to\infty}\widehat{\mathcal{R}}(\phi_n)= \widehat{\lambda}_1.
$$
Thus, by using Fatou's Lemma and the fact that $\widehat{\lambda}_1<0$, we get that $\phi$, denoted from now on $\widehat{\varphi}_1$, is the desirable minimizer of \eqref{la11}.\smallskip
\item[{\em iii)}] Let $\widehat{\varphi}$ be another a minimizer of \eqref{la11}. Follows, almost verbatim, the computations done for proving Proposition \ref{saleg-1} {\it iii)} in the transformed problem, we get that $\widehat{\varphi}(x)=\widehat{\varphi}_1(x)$ a.e. $x\in\mathbb{R}$, and, therefore, we conclude that $\widehat{\lambda}_1$ is simple as wanted.
\end{enumerate}
\end{proof}
{{
\textcolor{black}{To conclude this subsection we want to highlight that if} we consider \textcolor{black}{$n\geq 1$ an arbitrary natural number and} $\widehat{M}$ a $(n-1)$-dimensional subspace spanned by the eigenfunctions corresponding to the eigenvalues $\widehat{\lambda}_1, \dots, \widehat{\lambda}_{n-1}$, {\it i. e.} {{$ \widehat{M}=\mathrm{span}\{\widehat{\varphi}_1,\ldots ,\widehat{\varphi}_{n-1}\} $}}, \textcolor{black}{with}%
\begin{equation*}
\widehat{\lambda}_n:=\inf_{\scriptsize
{\begin{array}[c]{cc} 
v \in  \widehat{\mathcal{H}}(\mathbb{R})\setminus\{0\} \\
v\perp \widehat{M} 
\end{array}}}
\frac{\widehat{\mathcal{R}}(v)}{\|v\|^2_{L^2(\mathbb{R})}},
\end{equation*}
\textcolor{black}{similar results to those obtained in Propositions \ref{propovalprop2}-\ref{propola-i} can be proved for $\widehat{\mathfrak{L}}$.} That is, \textcolor{black}{if $\widehat{\lambda}_n<\mathcal{A}_{N,s}$ then $\widehat{\lambda}_n$ is an eigenvalue associated to an eigenfunction}.

For that we show \textcolor{black}{the following fundamental auxiliary.}
\begin{lemma}\label{lema4}
Let $\lambda\in \mathbb{R}$. 
Then, there exists a function $\widehat{v}$ that solves
$$
\widehat{\mathfrak{L}}\widehat{v}= \lambda \widehat{v}\quad \mbox{ in }\mathbb{R},
$$
if and only if there exists  a radial function $v$ that solves 
$$
\mathfrak{L} v=\frac{\lambda}{|x|^{2s}} v\quad \mbox{in } \mathbb{R}^N.
$$
\end{lemma}
\begin{proof}
Let $v$ be a radial function such that $\mathfrak{L}v=(-\Delta )^s v -p^*_{\alpha,s} |x|^{\alpha}U_{\alpha,s}^{p^*_{\alpha,s}-1}v=\frac{\lambda}{|x|^{2s}}v$ for every $x\in\mathbb{R}^{N}$. 
Without loss of generality, for every $r>0$ we may consider 
$$
v(x)=r^{\beta}w(r)\quad \mbox{and}\quad U_{\alpha,s}(r)=r^{\beta}W_{\alpha,s}(r),
$$
where $\beta=-\frac{N-2s}{2}$. 
Therefore, since
$$
(-\Delta )^s v= r^{\beta-2s} (\mathcal{S}w(r)+\mathcal{A}_{s,N} w(r)),\quad r>0,
$$
where{{
$$
\mathcal{S} v(r):=C_{N,s}\int_0^{\infty} \int_{\mathbb{S}^{N-1}}  {\rho}^{N-1+\beta} \frac{v(r)-v(r {\rho})}{|1+ {\rho}^{\,2}-2  {\rho} 
\langle \theta,\sigma\rangle |^{\frac{N+2s}{2}}} \, d\sigma\,d {\rho},\quad r>0,
$$
and}} $\mathcal{A}_{s,N}$ is given by (\ref{AA}), we get
$$
r^{\beta-2s} (\mathcal{S}w(r)+\mathcal{A}_{s,N} w(r))-p^*_{\alpha,s} r^{\alpha+\beta+\beta(p^*_{\alpha,s}-1)}W_{\alpha,s}^{p^*_{\alpha,s}-1}w(r) =\lambda r^{\beta-2s} w(r),\quad r>0.
$$
That is,
$$
\mathcal{S}w(r)+\mathcal{A}_{s,N} w(r)-p^*_{\alpha,s} W_{\alpha,s}^{p^*_{\alpha,s}-1}w(r) =\lambda  w(r),\quad r>0.
$$
Doing now the Emden-Fowler transformation $r=e^{\kappa}$ we obtain for $w(e^{\kappa})=\widehat{v}(\kappa)$, 
$$
\mathcal{T}_{s} \widehat{v}(\kappa) +(\mathcal{A}_{s,N} -p^*_{\alpha,s} \widehat{U}_{\alpha,s}^{p^*_{\alpha,s}-1} )\widehat{v}(\kappa) =\lambda  \widehat{v}(\kappa),\quad r>0,
$$
where the operator $\mathcal{T}_s$ is given by (\ref{TT}) and 
$$
\widehat{U}_{\alpha,s}(\kappa)=W_{\alpha,s}(e^{\kappa}).
$$
The reverse follows in the same way so we omit the details and the proof finishes. 
\end{proof}}}

\subsection{On the value of the second eigenvalue}
Let us consider the energy functional
$$
J(u)=\frac{{{C_{N,s}}}}{2}\iint_{\mathbb{R}^2} (u(\tau)-u(\kappa))^2\mathcal{K}(t-s)\,\,d\tau\,d\kappa+\frac{\mathcal{A}_{s,N}}{2}\int_{\mathbb{R}}u^2\,d\kappa -\frac{1}{p^*_{\alpha,s}+1}\int_{\mathbb{R}} u^{p^*_{\alpha,s}+1}\,d\kappa
$$
and the {{closed}} subspace 
$$
S:=\left\{ u\in \widehat{\mathcal{H}}\,:\, \langle J'(u),u\rangle=0\right\}:=\left\{ u\in \widehat{\mathcal{H}}\,:\, G(u)=0\right\}.
$$
We will say that
%
$$
v\in\widehat{\mathcal{H}} \mbox{ is a ground state solution if }\inf_{u\in S} J(u) =J(v).
$$
It is important to bear in mind that, by using symmetric decreasing rearrangement we may assume that a ground state solution is even and decreasing (see Corollary  2.3 of \cite{lieb}). 
Moreover, we also notice that $\widehat{U}_{\alpha,s}$ is a ground state solution (see \cite{Felmer, Barrios-Quaas}). The main ideas behind the existence of this ground state used in \cite{Felmer} are the Ekeland variational principle and the compactness concentration done by Beresticky and Lions given the existence of solutions of
\begin{equation}\label{prr}
\mathcal{T}_{s}v+ \mathcal{A}_{s,N} v = v^{p}\quad \mbox{ in }\mathbb{R},
\end{equation}
for any $1<p<2^{**}_s-1$, so that, in particular for \eqref{transformado}. 

We establish now the fundamental result regarding with the second eigenvalue of the transformed problem.
\begin{lemma}\label{lanuestravale}
If we define
\begin{equation}\label{la22}
\widehat{\lambda}_2:=\inf_{\scriptsize
{\begin{array}[c]{cc}\widehat{v}\in \widehat{\mathcal{H}}(\mathbb{R})\setminus\{0\}\\ \widehat{v}\perp \widehat{\varphi}_1\end{array}}} \frac{\widehat{\mathcal{R}}(\widehat{v})}{\|\widehat{v}\|_{L^2(\mathbb{R})}^{2}},
\end{equation}
where 
$$
\widehat{v}\perp \widehat{\varphi}_1:=\int_{\mathbb{R}} \widehat{v}\widehat{\varphi}_1=0,
$$
then $ \widehat{\lambda}_2\geq 0.$ That is, $N(\widehat{\mathfrak{L}}_{\widehat{U}_{\alpha,s}})=1$.
\end{lemma}

\begin{proof} 
The proof is quite standard in Calculus of Variation (see for example, \cite{Pacella}). 
In fact, let us consider
\begin{equation}\label{rat}
{\widehat{\mathcal{R}}(\varphi)=J''(\widehat{U}_{\alpha,s})[\varphi,\varphi]\quad \mbox{for all }\varphi\in T:=\left\{w\in \widehat{\mathcal{H}}\,:\, \langle G'(\widehat{U}_{\alpha,s}),w\rangle=0\right\}}.\end{equation}
To obtain the desirable sign of the second eigenvalue will be enough to prove that {{
\begin{equation}
\label{claim}J''(\widehat{U}_{\alpha,s})[\varphi,\varphi]\geq 0, \quad \varphi\in T.
\end{equation}
Indeed,}} if \eqref{claim} is satisfies then, since $\widehat{\lambda}_2$ may be characterized as{{
$$
\widehat{\lambda}_2= \inf_{W\subset \widehat{\mathcal{H}}(\mathbb{R})} \sup_{w\in W} \frac{\widehat{\mathcal{R}}(w)}{\|w\|_{L^2(\mathbb{R})}^{2}},
$$
with $W$ any vector space of dimension two.  
Since by the fact that $\langle G'(\widehat{U}_{\alpha,s}),\widehat{U}_{\alpha,s}\rangle \neq 0$, $T$ has codimension 1, we conclude by taking
$$
W:=\mathrm{span}\{\widehat{U}_{\alpha,s},\varphi\},
$$
with}} $\varphi\in T$. So let us prove \eqref{claim}. 
For that we consider $\varphi\in T$. 
Then there exists a curve $\gamma:(-\varepsilon,\varepsilon)\rightarrow S$ such that $\gamma'(0)=\varphi$ and $\gamma(0)=\widehat{U}_{\alpha,s}$. 
Observe that by a Taylor expansion, since $\widehat{U}_{\alpha,s}$ minimizes the functional in $S$, we get
$$
0\leq J(\gamma(t))-J(\gamma(0))=\langle J'(\widehat{U}_{\alpha,s}),\varphi\rangle t+J''(\widehat{U}_{\alpha,s})[\varphi,\varphi]\frac{t^2}{2}+\langle J'(\widehat{U}_{\alpha,s}),\gamma''(0)\rangle \frac{t^{2}}{2}+o(t^2).
$$
Thus, 
$$
J''(\widehat{U}_{\alpha,s})[\varphi,\varphi](1+o(1))\geq 0,\quad \mbox{where $o(1)\rightarrow 0$ as $t\rightarrow 0$},
$$
so that, the claim, and so the whole proof, follows. 
\end{proof}

\begin{remark}\label{rema}
{\it i)} {{Let $u\in {H^s_{\mathrm{rad}}(\mathbb{R}^N)}$ and $\widehat{u}\in \widehat{\mathcal{H}}$}} the corresponding transformed function. 
Then
$$ 
{\int_{\mathbb{R}} \widehat{u}\widehat{\varphi}_1=0 \Leftrightarrow \int_{\mathbb{R}^N} \frac{u \varphi_1}{|x|^{2s}}=0.}
$$
{\it ii)} If there exists $\lambda_1\leq\lambda_{k_0}<0$ a negative eigenvalue of $\mathfrak{L}$,  then $\lambda_{k_0}=\lambda_1$. 
Indeed, if $\lambda_{k_0}<0$ would be an eigenvalue associated to the eigenfunction $\phi_2$, by Lemma \ref{lema4}, it is clear that $\widehat{\phi}_2$ will be an eigenfunction of  $\widehat{\mathfrak{L}}$ associated to the second eigenvalue. 
By the previous proposition we get that $\lambda_{k_0}\geq 0$  that is a contradiction.
\end{remark}
The last step to prove the Theorem \ref{nodegenerancia_extendido} is the next lemma that corresponds to the adapted version in our framework of \cite[Lemma C.3]{Frank-Lenzmann}. 
The tools that we will use to prove it are different to the ones developed in \cite{Frank-Lenzmann} because the nonlocal operator $\mathcal{T}_{s}$ given in \eqref{TT} is not directly define by a semigroup that match well with the Fourier Transform, that is the key ingredient of the proof of R. Frank and E. Lenzmann.{{

\begin{lemma}\label{final}
 If $\widehat{\varphi}_2\in  \widehat{\mathcal{H}}$ is an odd eigenfunction associated to the eigenvalue $\widehat{\lambda}_2=0$ then, $\widehat{\varphi}_2(x)> 0$ for $x > 0$ (after replacing 
$\widehat{\varphi}_2$ by $-\widehat{\varphi}_2$ if necessary) and $\widehat{\lambda}_2=0$  is a simple eigenvalue in the set of odd functions. 
\end{lemma}
\begin{proof} 
Since by assumption we have that $\widehat{\lambda}_2=0$ is attainted by $\widehat{\varphi}_2$ an odd function, we get that
%
$$
\widehat{\lambda}_2=0=\inf_{\scriptsize {\begin{array}[c]{cc}\widehat{v}\in \widehat{\mathcal{H}}\setminus\{0\}\\ \widehat{v}\perp \widehat{\varphi}_1\\
\widehat{v} \,\mbox{odd}\end{array}}} \frac{\widehat{\mathcal{R}}(\widehat{v})}{\|\widehat{v}\|_{L^2(\mathbb{R})}^{2}}.
$$
%
{{Since $\widehat{V}=-p^*_{\alpha,s}\widehat{U}_{\alpha,s}^{p^*_{\alpha,s}-1}$ is an even function clearly
\begin{eqnarray*}
\widehat{\mathcal{R}}(\widehat{v})&=&2 C_{N,s}\int_0^\infty \int_0^\infty (\widehat{v}(\kappa)-\widehat{v}(\tau))^2\mathcal{K}(\kappa-\tau)+(\widehat{v}(\kappa)+\widehat{v}(\tau))^2\ \mathcal{K}(\kappa+\tau)d\tau\,d\kappa\\
&+&2\mathcal{A}_{s,N}\int_{0}^{\infty}\widehat{v}^2d\tau+2\int_0^\infty \widehat{V}(\tau)\widehat{v}^2d\tau.
\end{eqnarray*}}}
In order to affirm that, after replacing sign, $\widehat{\varphi}_2\geq 0$, we will prove that
\begin{eqnarray}\label{true}
A\!&\!\!:=\!\!&\!(|\widehat{v}(\kappa)|-|\widehat{v}(\tau)|)^2\mathcal{K}(\kappa-\tau)+(|\widehat{v}(\kappa)|+|\widehat{v}(\tau)|)^2\ \mathcal{K}(\kappa+\tau)\nonumber\\
\!&\!\leq\!\!&\! (\widehat{v}(\kappa)-\widehat{v}(\tau))^2\mathcal{K}(\kappa-\tau)+(\widehat{v}(\kappa)+\widehat{v}(\tau))^2\ \mathcal{K}(\kappa+\tau):=B,
\end{eqnarray}
for every $(\kappa,\tau) \in (0,\infty)^2$. 
Since the previous inequality is trivial if we consider $\widehat{v}(\tau)\widehat{v}(\kappa)\geq0$, in what follows we will work with $\kappa,\, \tau\geq 0$ such that $\widehat{v}(\tau)\widehat{v}(\kappa)< 0$. 
In this case it is clear that
$$
(\widehat{v}(\kappa)+\widehat{v}(\tau))^2=(|\widehat{v}(\kappa)|-|\widehat{v}(\tau)|)^2
$$
and
$$
(\widehat{v}(\kappa)-\widehat{v}(\tau))^2= (|\widehat{v}(\kappa)|+|\widehat{v}(\tau)|)^2,
$$
%
so that
$$
A\leq (\widehat{v}(\kappa)+\widehat{v}(\tau))^2\mathcal{K}(\kappa-\tau)+(\widehat{v}(\kappa)-\widehat{v}(\tau))^2\ \mathcal{K}(\kappa+\tau)\leq B,
$$
where in the last inequality we have used that $\mathcal{K}$ is even, the fact that $\mathcal{K}(\cdot)$ is strictly-decreasing and $(\kappa+\tau)> (\kappa-\tau)$ for $\kappa,\tau>0$. 
So that \eqref{true} follows and we can affirm that if $\widehat{\varphi}_2$ is a minimizer then $\widehat{\varphi}_2\geq 0$ in $(0,\infty)$. 
Applying now a {strong maximum principle} (see \cite[Proposition 3.6]{Manuel}) it follows that  $\widehat{\varphi}_2>0$.

The proof of the uniqueness of $\widehat{\varphi}_2$ follows as in Proposition \ref{saleg-1} {\it iii)} doing a normalization of the odd functions with respect to the $L^{2}(\mathbb{R}^{+})$ norm.
\end{proof}}}

\subsection{Proof of Theorem \ref{nodegenerancia_extendido}.}

Let us consider the orthogonal decomposition
$$
L^{2}(\mathbb{R})=L_{even}^{2}(\mathbb{R})\oplus L_{odd}^{2}(\mathbb{R}).
$$
Since, as we commented before, without loss of generality, we can consider that $\widehat{U}_{\alpha,s}$ is even then $\mathfrak{L}$ leaves the subspaces $L_{even}^{2}(\mathbb{R})$ and $L_{odd}^{2}(\mathbb{R})$ invariant so we will analyze these subspaces separately. 
First of all we notice that (see \cite{Barrios-Quaas}), since for $|x|>>1$,
$$
\widehat{U}'_{\alpha,s}(|x|)\leq C(s)\frac{|x|^{-1}}{|x|^{\frac{N-2s}{2}}},
$$
then $\widehat{U}'_{\alpha,s}(|x|)\in L^2(\mathbb{R})$.  Moreover since it is clear that $\widehat{\mathfrak{L}}(\widehat{U}'_{\alpha,s})=0$, then by Proposition \ref{final} we get that $\widehat{U}'_{\alpha,s}$ is the unique odd eigenvalue associated associated to $\widehat{\lambda}_2$. 
That is, 
$$
{{{\rm Ker}\,\widehat{\mathfrak{L}}|_{L_{odd}^{2}(\mathbb{R})}=\mathrm{span}\{\widehat{U}'_{\alpha,s}\}.}}
$$
To conclude, we have to prove that
$$
{\rm Ker}\,\widehat{\mathfrak{L}}|_{L_{even}^{2}(\mathbb{R})}=\{0\},
$$
that is, that {{$\widehat\lambda_2=0$}} can not be an eigenvalue with  even eigenfunction. 
Let us suppose by contradiction that there exists $\widehat{\phi}_{2}\in L_{even}^{2}(\mathbb{R})$ such that $\widehat{\mathfrak{L}}\widehat{\phi}_{2}=0$. 
First of all we notice that
\begin{equation}\label{afirmacion_clave}
\mbox{$\widehat{\phi}_{2}$ changes its sign exactly one in $\{x>0\}$.}
\end{equation}
Indeed, since by Lemma \ref{lema4} and Remark \ref{rema} {\it ii)} we get that
$$
\phi_2(r)=r^{-\frac{N-2s}{2}} \widehat{\phi}_{2}(\ln\, r),\, r>0,
$$
is a radial eigenfunction of $H=(-\Delta)^s{+}V $ associated to the zero eigenvalue, if $\widehat{\phi}_{2}$ would change sign twice in $\{x>0\}$ then $\phi_2$ would do it four times in the positive axis contradicting Proposition \ref{clave}.  
Thus, since \eqref{afirmacion_clave} is true, we get that there exists $r^*\in (0,+\infty)$ such that 
\begin{equation}\label{namast}
\mbox{$\phi_2 \geq0$ in $(0,r^*)$, $ \phi_2(r^*)=0 $ and  $\phi_2(r)\leq 0$ in $[r^*,+\infty)$,}
\end{equation}
(after a posible change $\phi_2$ by $-\phi_2$).
Since 
%
$$\mathcal{T}_{s}\widehat{U}_{\alpha,s}+ \mathcal{A}_{s,N} \widehat{U}_{\alpha,s} = \widehat{U}_{\alpha,s}^{p^*_{\alpha,s}}=e^{p^*_{\alpha,s}\ln \widehat{U}_{\alpha,s}},$$
differentiating with respect to $p^*_{\alpha,s}$ we get that {{$v:={\partial \widehat{U}_{\alpha,s}}/{\partial p^*_{\alpha,s}}$}} satisfies
\begin{equation}\label{xxt}
\widehat{\mathfrak{L}}(v)=  \widehat{U}_{\alpha,s}^{p^*_{\alpha,s}} \ln \widehat{U}_{\alpha,s}.
\end{equation}
{
Notice that existence for \eqref{prr} holds for $p$ close to  $p^*_{\alpha,s}$ as mentioned above.} Thus, denoting by $r=|x|$, if we define
$$
w(x)=w(r):=
\begin{cases}
(p^*_{\alpha,s}-1)\left(\frac{\ln \widehat{U}_{\alpha,s}(r)}{\ln \widehat{U}_{\alpha,s}(r^*)}-1\right)\widehat{U}_{\alpha,s}^{p^*_{\alpha,s}}(r), &\text{if }\widehat{U}_{\alpha,s}(r^*)\not =1,\\[8pt]
\widehat{U}^{p^*_{\alpha,s}}_{\alpha,s}(r)\ln \widehat{U}_{\alpha,s}(r), &\text{if }\widehat{U}_{\alpha,s}(r^*) =1,
\end{cases}
$$
by \eqref{flower} and \eqref{xxt} we get that $w$ in the rank of $\widehat{\mathfrak{L}}$. 
Moreover since ${p^*_{\alpha,s}}>1$ and $\widehat{U}_{\alpha,s}(r)$ is positive and strictly decreasing, we deduce that
$$
\mbox{$w \geq0$ in $(0,r^*)$, and  $w(r)\leq 0$ in $[r^*,+\infty)$.}
$$
By \eqref{namast} this implies that
%
$$
\int_\mathbb{R}  \phi_2(\tau) w(\tau) d\tau >0,
$$
getting a contradiction with the standard Fredholm alternative theory.
 
\section{Uniqueness}

As was commented in the Introduction, throughout this section we will assume $N\geq 3$. Notice that 
by Remark \ref{auxi}, the hypothesis $N(\widehat{\mathfrak{L}}_{Q_{s_0}})=1$ in Theorem \ref{unicidad1} may be changed by $N_{even}(\widehat{\mathfrak{L}}_{Q_{s_0}})=1$, that is, the Morse index of 
$$
\widehat{\mathfrak{L}}_{Q_{s_0}}=\mathcal{T}_{s_0}+ \mathcal{A}_{s_0,N}- p\, Q_{s_0}^{p-1},
$$
acting on $L^2_{even}(\mathbb{R})$ is constant and equal to one.

\subsection{The local branch}

Let $\textcolor{black}{Q(s)=Q_s}$ be a solution of \eqref{cc}.  It is clear that  $Q_s\in \widehat{\mathcal{H}}^s(\mathbb{R})$ \textcolor{black}{defined in \eqref{el_espacio}}, so that $Q_s\in L^{r}(\mathbb{R})$ for every $1<r\leq {{2_s^{**}-1}}$ ( or $\infty$) if $s\leq \frac{1}{2}$ (or $s \geq\frac{1}{2}$).  Since by the fast decay estimates obtained in \cite{Barrios-Quaas} (see Remark 4.7 ii) there) we can use moving planes method in a similar way as in \cite{Felmer-Ying} in the transformed equation\eqref{cc}, to prove that $Q$ is symmetric with respect to a point, we could assume, without loss of generality that $Q$ is, up to translation, an even function. Therefore we define the space
$$
X^p:=\{ f\in L^2(\mathbb{R})\cap L^{p+1}(\mathbb{R})\,| \, f\mbox{ is even }\},\quad p<2_s^{**}-1,
$$
with a norm 
$$
\|f\|_{X^p}:=\|f\|_{L^2(\mathbb{R})}+ \|f\|_{L^{p+1}(\mathbb{R})}.
$$  
%
We highlight that the space $X^p$ is $s$-independent. Let us establish now the existence and local uniqueness of a branch $Q_s$ around a solution $Q_{s_0}$.
%
\begin{lemma}\label{unicidad_1}
Let fix $0<s_0<1$, $0<p<2_{s_0}^{**}-1$ and let us assume that
\begin{equation}\label{hip}
\mbox{$Q_{s_0}\in X^{p}$ is a solution of \eqref{cc} with $s=s_0$ such that $N_{even}(\widehat{\mathfrak{L}}_{Q_{s_0}})=1$}.
\end{equation}
%
Then, for some $\delta>0$ there exists a map $Q\in C^1([s_0, s_0+\delta]; X^{p})$ such that
\begin{itemize}
\item [i)]  For every $s\in [s_0, s_0+\delta]$, $(Q_s, \mathcal{A}_{s,N})$ is a solution of \eqref{cc} 
\item [ii)]There exists $\varepsilon>0$ such that $(Q_s, \mathcal{A}_{s,N})$ is the unique solution of \eqref{cc} for $s\in [s_0, s_0+\delta]$ in the set $\{Q\in X^{p}: \|Q-Q_{s_0}\|_{X^p}<\varepsilon\}$.
\end{itemize}
\end{lemma}

\begin{remark}\label{gato}
By Theorem \ref{nodegenerancia_extendido} (see Remark \ref{importante}) we notice that the hypothesis \eqref{hip} guarantee that the linearized operator $\widehat{\mathfrak{L}}_{Q_{s_0}}$ has a trivial kernel on $L^2_{even}(\mathbb{R})$ so that has a bounded inverse on $X^p$.
\end{remark}
\begin{proof}
The proof follows by using the implicit function theorem by defining the map
$$
F(Q, s):=Q-\frac{1}{\mathcal{T}_{s}+ \mathcal{A}_{s,N}}Q^{p}.
$$
Indeed, by using the fact that 
$$
(\mathcal{T}_{s}+ \mathcal{A}_{s,N})v(t)=e^{t\left(\frac{N+2s}{2}\right)}(-\Delta)^s\left(e^{-t\left(\frac{N-2s}{2}\right)}v(t)\right),
$$
and the results of \cite{Frank-Lenzmann} and \cite{Frank-Lenzmann-Silvestre} we get that $F(Q,s)$ is well defined and $C^1$. Moreover, there exists the inverse of $\partial_{Q}F$ at $(Q_{s_0},s_0)$, that is, for every $f\in X^p$ there exists an unique $\eta\in X^p$ such that $\partial_{Q_{s_0}}F\eta=f$ or, what is the same, the unique solution of
$$
\left(1-\frac{p}{\mathcal{T}_{s_0}+ \mathcal{A}_{s_0,N}}Q_{s_0}^{p-1}\right)\eta=0,
$$
is $\eta\equiv 0$. Indeed if $\eta\in X^p$ satisfies the previous equality we get that $\eta\in \mathrm{Ker}\,\widehat{\mathfrak{L}}_{Q_{s_0}}$ so we conclude by \eqref{hip} and the fact that $\eta$ is an even function (see Remark \ref{gato}). 
Once we have proved that  $\partial_{Q}F$ is invertible at $(Q_{s_0},s_0)$ the Implicit Function Theorem applied to the map $F(Q,s)$ at $(Q_{s_0}, s_0)$ imply the conclusions {{{ i)-ii)} by taking $\delta>0$ and $\varepsilon>0$ small enough.}}
\end{proof}

\subsection{Global continuation}

Let us noiw consider $Q_s\in X^p$ a local branch defined for $s\in [s_0, s_0+\delta]$ established in Lemma \ref{unicidad_1} for some fixed $0<s_0<1$ and $0<p<2_{s_0}^{**}-1$ (or $<\infty$ if $s_0\geq 1/2$). 
We consider now the maximal extension of the branch $Q_s$ for $s\in [s_0, s^*)$ {{where
$$
s^*:=\sup\{ s_0<\tilde s<1 \,|\, Q_s\in C^1([s_0, \tilde s); X^p) \mbox{ satisfies \eqref{hip} for every $s\in [s_0, \tilde s)$}<1\}.
$$
The}} objective is to prove that $s^*=1$ under some suitable hypothesis on $Q_{s_0}$. That is, get the following
\begin{proposition}\label{unicidad_2}
Let fix $0<s_0<1$ and $0<p<2_{s_0}^{**}-1$ (or $<\infty$ if $s_0\geq 1/2$) and $Q_{s_0}\in X^{p}$ satisfying \eqref{hip}. 
If we assume that $Q_{s_0}>0$ then the maximal branch $Q_s\in C^1([s_0, s^*); {X^p})$ its extend to $s^*=1$. 
Moreover,
$$
Q_s\to Q_{1}\mbox { in $L^{2}(\mathbb{R})\cap L^{p+1}(\mathbb{R})$ when $s\to 1^{-}$,}
$$
where $Q_1>0$ is the unique solution of
$$
\mathcal{T}_1Q_{1}+ \mathcal{A}_{1,N} Q_{1} = Q_1^{p^*_{\alpha,1}-1}\quad \mbox{ in }\mathbb{R},
$$
where
$$
\mathcal{T}_1 w:=\lim_{s\to 1^{-}} \mathcal{T}_{s}w= - w''
$$
and
$$
\mathcal{A}_{1,N}=4\frac{\Gamma^2\left(\frac{N+2}{4}\right)}{\Gamma^2\left(\frac{N-2}{4}\right)}{=\left(\frac{N-2}{2}\right)^{2}}.
$$
\end{proposition}

To establish the previous proposition we need some auxiliary results. 
The first one affirms that the positivity property holds along the maximal branch.

\begin{lemma}\label{positividad}
If $Q_{s_0}(|x|)>0$, then $Q_s(|x|)>0$ for all $s\in [s_0, s^*)$.
\end{lemma} 
\begin{proof} 
{{Case $0<s_0<{1}/{2}$}}.
{\underline {1) Open property}:} Let us prove that the sign changing is a close property. 
Let us suppose that $Q_s$ changing sign for every $s_0\leq s<\widetilde{s}<s^*$ and let us prove that $Q_{\widetilde{s}}$ also does it. 
For that we consider $\{s_n\}\subseteq [s_0,s^*)$ such that $s_n\to\widetilde{s}$ with $Q_{s_n}$ sign changing. 
By testing with $Q_{s_n}$ we get that  {{$\|Q_{s_n}\|_{\widehat{\mathcal{H}}^{s_n}}\leq C$, so in particular also $\|Q_{s_n}\|_{{\widehat{\mathcal{H}}^{s_0}}}\leq C$.}} Thus $Q_{s_n} \to Q_{\widetilde s}$ a.e. in $\mathbb{R}$. 
If we suppose that $Q_{\widetilde s}>0$ then by testing the equation of $Q_{s_n}$ by  $Q_{s_n}^-:=\min\{Q_{s_n}, 0\}$, since
$$
(Q_{s_n}(t)-Q_{s_n}(\widetilde{t}))(Q_{s_n}^{-}(t)-Q_{s_n}^{-}(\widetilde{t}))\geq (Q_{s_n}^{-}(t)-Q_{s_n}^{-}(\widetilde{t}))^2,
$$
we obtain 
{$$
C\|Q_{s_n}^-\|_{L^{\frac{2}{1-2s_{0}}}(\mathbb{R)} }
\leq \left(\int_{\{\mathbb{R}\cap \{Q_{s_n}<0\} }|Q_{s_n}|^{p\frac{2}{1+2s_{0}}}\right)^{\frac{1+2s_{0}}{2}} \|Q_{s_n}^-\|_{L^{\frac{2}{1-2s_{0}}}(\mathbb{R)}}.
$$
Since $p$ is subcritical the previous inequality implies a contradiction with the fact that  $Q_{s_n}^-\to 0$ a.e. in $\mathbb{R}$.}
\\
\underline {2) Close property}: Let us suppose that $Q_s>0$ for every $s_0\leq s<\widetilde{s}<s^*$ and we will obtain that $Q_{\widetilde{s}}>0$. 
Indeed, since doing as before we get that $Q_{s_n} \to Q_{\widetilde s}$ a.e in $\mathbb{R}$ with $s_n\to\widetilde{s}$, then $Q_{\widetilde s}\geq 0$. 
Moreover, since $Q_{\widetilde s}$ is a weak solution of \eqref{cc} with $s=\widetilde{s}$ by \cite[Proposition 3.6]{Manuel} we know that $Q_{\widetilde s}>0$ or $Q_{\widetilde s}\equiv 0$. 
However, the second option it is not possible due to the fact that, doing similarly as before, {testing the equation of $Q_{s_n}$ with $Q_{s_n}$ we get that $\int_{\mathbb{R}}|Q_{s_n}|^{p\frac{2}{1+2s_0}} $ is bounded from below. 

{{Case ${1}/{2}\leq s_0<1$. This case is more simple and use the $L^\infty$ norm  given by a Morrey type estimate.}}}
\end{proof}

The last fundamental step to obtain Proposition \ref{unicidad_2} will be to get the strong convergence of the branch in $L^2(\mathbb{R})\cap L^{p+1}(\mathbb{R})$ by using some a priori bounds. That is.
\begin{lemma}\label{convergencia}
Let $\{s_n\}\subseteq [s_0,s^*)$ such that $s_n\to s^*$ and $Q_{s_n}$ be a solution of \eqref{cc} with $s=s_n$. 
Then, up to subsequence, 
$$
Q_{s_n}\to Q_{s^*}\mbox{ in $L^2(\mathbb{R})\cap L^{p+1}(\mathbb{R})$,}
$$
with $Q_{s^*}$ satisfying \eqref{cc} with $s=s^*$.
\end{lemma}

The proof of the previous result follows verbatim the one of \cite[Lemma 8.4]{Frank-Lenzmann-Silvestre} (see also \cite[Lemma 5.7]{Frank-Lenzmann}) once we get a suitable a priori bounds that will allow us to obtain the global convergence in the Hilbert space $L^2(\mathbb{R})$ instead of the local one commonly obtained by using the Ascoli-Arzela Theorem. 
Thus the key step will be to establish the next
\begin{lemma} 
If {$Q_{s_0}(|x|)>0$}, then $Q_s(|x|)\leq C$ for all $s \in [s_0, s^*)$. Therefore 
$$Q_s(|x|)\leq C\frac{1}{|x|^{\frac{N-2s}{2}}}, \, \mbox{ for every $s\in [s_0, s^*)$ and $|x|\gg 1$.}$$
\end{lemma} 
\begin{proof} 
To show the universal boundedness of the solutions we will use the well-known blow-up  method of Gidas and Spruck in our nonlocal framework. 
Indeed, let us consider $s \in [s_0, s^*)$ and let us suppose, by contradiction, that there exists a sequence $s_n\to  s$  as $n \to \infty$ such that the solution $Q_{s_n}$  satisfies $0\leq M_n:=\| Q_{s_n}\|_{L^\infty(\mathbb{R})}= Q_{s_n}(0) \to +\infty$.
We define
$$
z_n(x)= \frac{Q_{s_n}(M_n^{-\frac{p-1}{2s_n}}x)}{M_n}, \quad x\in \mathbb{R}.
$$
and $\varepsilon_n=M_n^{-\frac{p-1}{2s_n}}$. Following the ideas done in \cite[Lemma 4.3]{Manuel} we get that
$$ 
 {{C_{N,s}}}\int_{\mathbb{R}}\varepsilon_n (z_n(\kappa))-z_n(\tau))\mathcal{K}_n(\varepsilon_n(\kappa-\tau))\,d\tau+ {{\mathcal{A}_{N,s} }}z_n(\kappa)=\varepsilon_n^{-2s_n}z_n^p.
$$
Since $K_n(t)$ behaves like $|t|^{1+2s_n}$ for $t$ close to $0$, it follows that 
$$
 {{C_{N,s}}}\int_{\mathbb{R}}\varepsilon_n (z_n(\kappa))-z_n(\tau))\mathcal{K}_n(\varepsilon_n(\kappa-\tau))\,d\tau= c\varepsilon_n^{-2s_n}(-\Delta)^{s_n}z_n(\kappa)+o(1).
$$
Thus, since $z_n \to z_\infty$ uniformly on compact sets we get that $z_\infty$ satisfies
$$
(-\Delta)^{s} z_\infty=z_\infty^p,
$$
and $z_\infty(0)=1$ that contradicts \cite[Theorem 1.2]{Felmer1} when $s\ge \frac{1}{2}$ and \cite[Theorem 4]{ZZ} (see also \cite{CLO}) if $s<\frac{1}{2}$ for dimension $N=1$. 
Therefore $Q_s(|x|)\leq C$ for all $s \in [s_0, s^*)$ and, by using the method developed in \cite{Barrios-Quaas}, we conclude that $Q_s(|x|)\leq C|x|^{-(N-2s)/2} $for all $s\in [s_0, s^*)$.
\end{proof}

Now we are able to get the
\begin{proof}[Proof of Proposition \ref{unicidad_2}] 
The proof follows like \cite[Lemma 8.5]{Frank-Lenzmann-Silvestre} (that is based on \cite[Proposition 5.2]{Frank-Lenzmann}) using the lower semi-continuity of the Morse index,  Lemma \ref{positividad}, {Theorem \ref{nodegenerancia_extendido}} (see Remark \ref{importante}) and Lemma \ref{convergencia}.
\end{proof}

\subsection{Proof of Theorem \ref{unicidad1}}

Let fix $0<s_0<1$ and $p<2^{**}_{s_0}-1$ and let us suppose by contradiction that there exist two positive  solutions $ Q_{s_0}\neq \widetilde{Q}_{s_0}$ of problem 
$$
\mathcal{T}_{s_0}v+ \mathcal{A}_{s_0,N} v = v^{p}\quad \mbox{ in }\mathbb{R},
$$
with $N(\widehat{\mathfrak{L}}_{Q_{s_0}})=N(\widehat{\mathfrak{L}}_{\widetilde{Q}_{s_0}})=1$. 
As we have noticed before, by using a symmetric decreasing rearrangement, we can assume that the previous equality implies that $N_{even}(\widehat{\mathfrak{L}}_{Q_{s_0}})=N_{even}(\widehat{\mathfrak{L}}_{\widetilde{Q}_{s_0}})=1$. 
Therefore, since both solutions satisfy \eqref{hip}, by Proposition \ref{unicidad_2} there exist two global branches $Q_s\in C^1([s_0, 1); {X^p})$ and $\widetilde{Q}_s\in C^1([s_0, 1); {X^p})$ that solve
$$
\mathcal{T}_{s}v+ {\mathcal{A}_{s,N}} v = v^{p}\quad \mbox{ in }\mathbb{R}.
$$
%
On one hand, by the local uniqueness given in Lemma \ref{unicidad_1}, we have that the branches cannot intersect between them. 
On the other hand, by Proposition \ref{unicidad_2}, we know that
$$
Q_s\to Q_{1}\mbox { in $L^{2}(\mathbb{R})\cap L^{p+1}(\mathbb{R})$ when $s\to 1^{-}$,}
$$
$$
\widetilde{Q}_s\to {Q}_{1}\mbox { in $L^{2}(\mathbb{R})\cap L^{p+1}(\mathbb{R})$ when $s\to 1^{-}$,}
$$
with $Q_1>0$ is the unique solution of the equation
{
$$
\mathcal{T}_1Q_{1}+ \left(\frac{N-2}{2}\right)^{2} Q_{1} = Q_1^{p^*_{\alpha,1}-1}\quad \mbox{ in }\mathbb{R},
$$
where $\mathcal{T}_1:=\lim_{s\to 1^{-}} \mathcal{T}_{s}$.
} Using now that $Q_{1}$ {has a non-degenerate linearized operator} by the results of \cite{Gladiali-Grossi-Neves}, by using the Implicit Function Theorem in a {\it backward} way (see page 33 of \cite{Frank-Lenzmann}), we get that the branches $Q_s$ and $\widetilde{Q}_s$ must intersect in some $s_0\leq s<1$ giving a contradiction.\quad $\square$ 

\subsection{Proof of Theorem \ref{unicidad_original}}

By using Theorem \ref{unicidad1} we can establish the proof of our Theorem \ref{unicidad_original}. 
Indeed, let $Q$ be the transformation of $U$ under the Emden-Fowler change of variable by Lemma \ref{lema4} it is clear that  $N(\widehat{\mathfrak{L}}_{Q})=1$. Since, as we commented at the beginning of Subsection 3.1, we can assume that $Q$ is, up to translation, an even function, the uniqueness follows by applying Theorem \ref{unicidad1}.\quad$\square$\bigskip

We should note here that if the previous uniqueness Theorems \ref{unicidad1} and \ref{unicidad_original} were announced for ground state solutions
(see \cite{Frank-Lenzmann}), then the even property would trivially hold using symmetric rearrangement as mentioned above.\bigskip

{\bf Acknowledgements} 
S. A. was partially supported by FONDECYT Grants \# 1161635, \# 1181125 and \# 1171691 (Chile); 
B. B. was partially supported by AEI Grant MTM2016-80474-P and Ram\'on y Cajal fellowship RYC2018-026098-I (Spain). A. Q. was partially supported by FONDECYT Grant \# 1190282 and Programa Basal, CMM. U. de Chile.


\end{document}